\documentclass[12pt]{amsart}
\usepackage{amsmath}
\usepackage{amssymb}
\usepackage{amsfonts}
\usepackage{latexsym}
\usepackage{amscd}
\usepackage[mathscr]{euscript}
\usepackage{enumitem}
\usepackage{xy} \xyoption{all}
\usepackage{pdfsync}
\usepackage[active]{srcltx}

\newcommand{\N}{{\mathbb{N}}}

\newcommand{\Ninf}{{\mathbb{N}\cup\{\infty \}}}
\newcommand{\Z}{{\mathbb{Z}}}

\newcommand{\Oo}{{\mathcal{O}}}
\newcommand{\OAB}{{\mathcal{O}_{A,B}}}
\newcommand{\LAB}{{\Lambda_{A,B}}}

\newcommand{\Zplus}{\Z^+}
\newcommand{\OmA}{{\Omega_{A}}}
\newcommand{\SAB}{{\mathcal{S}(\LAB)}}

\newtheorem{lemma}{Lemma}[section]
\newtheorem{corollary}[lemma]{Corollary}
\newtheorem{theorem}[lemma]{Theorem}
\newtheorem{proposition}[lemma]{Proposition}
\newtheorem{remark}[lemma]{Remark}
\newtheorem{definition}[lemma]{Definition}
\newtheorem{example}[lemma]{Example}

\newtheorem{noname}[lemma]{}

   \usepackage{color}

\begin{document}

\title[Inverse semigroup crossed products]{Representing Kirchberg algebras as inverse semigroup crossed products}%

\author{Ruy Exel}
\address{Departamento de Matem\'atica, Universidade Federal de Santa Catarina, 88040-970 Florian\'opolis SC, Brazil}
\email{exel@mtm.ufsc.br}\urladdr{http://www.mtm.ufsc.br/~exel/}

\author{Enrique Pardo}
\address{Departamento de Matem\'aticas, Facultad de Ciencias\\ Universidad de C\'adiz, Campus de
Puerto Real\\ 11510 Puerto Real (C\'adiz)\\ Spain.}
\email{enrique.pardo@uca.es}\urladdr{https://sites.google.com/a/gm.uca.es/enrique-pardo-s-home-page/}\urladdr{http://www.uca.es/dpto/C101/pags-personales/enrique.pardo}


\thanks{The first-named author was partially supported by CNPq. The second-named author was partially supported by PAI III grants FQM-298 and P07-FQM-7156 of the Junta de Andaluc\'{\i}a, by the DGI-MICINN and European Regional Development Fund, jointly, through Project MTM2011-28992-C02-02 and by 2009 SGR 1389 grant of the Comissionat per Universitats i Recerca de la Generalitat de Catalunya.}




\begin{abstract}
In this note we show that a combinatorial model of Kirchberg algebras in the UCT, namely the Katsura algebras $\OAB$, can be expressed both as groupoid $C^*$-algebras and as inverse semigroup crossed products. We use this picture to obtain results about simplicity, pure infiniteness and nuclearity of $\OAB$ using different methods than those used by Katsura.
\end{abstract}
\maketitle

\section{Introduction}

$C^*$-algebras have been, from its very origin in the 1940s, a class of operator algebras deeply studied and analized. The complexity of these objects forced the search of classifying invariants to analyze them, mostly of K-Theoretical nature \cite{Black}.

In 1973 Elliott fixed the Classification Program: to classify separable nuclear $C^*$-algebras via K-Theoretic invariants (see \cite{Class} for an exhaustive account). This is still a very active field of research, with a great success in the case of purely infinite simple $C^*$-algebras. This class of algebras, originally introduced by Cuntz \cite{C}, have its most essential examples in the so-called Cuntz-Krieger algebras $\mathcal{O}_A$ (where $A\in M_n(\Z^+)$)  \cite{CK}, a proper subclass of separable nuclear purely infinite simple $C^*$-algebras. Cuntz-Krieger algebras have a pure combinatorial algebraic nature, and its classification, obtained by R\o rdam \cite{Rordam}, is related to symbolic dynamics via Franks' classification of irreducible essential subshifts of finite type \cite{Franks} (modulo the use of some sophisticated K-Theoretic tools). Simultaneously to R\o rdam's work, Kirchberg \cite{Kirch} and Phillips \cite{Phil} independently showed that separable nuclear purely infinite simple $C^*$-algebras satisfying the UCT (now called Kirchberg algebras) are classifiable using only $K_0$ and $K_1$ as invariants. This result extends R\o rdam's one, but it  is existential, and cannot be transferred to a combinatorial, algebraic context. \vspace{.2truecm}

On the combinatorial framework Kumjian, Pask, Raeburn and Renault \cite{KPR, KPRR} extended Cuntz-Krieger algebras to a biggest class, that of graph $C^*$-algebras (see \cite{Raeburn} for an exhaustive account). The milestone of this class is that its properties are faithfully reflected in the combinatorial structure of the graphs used to construct it. In the purely infinite simple case, they conform a class extending that of purely infinite simple Cuntz-Krieger algebras, but it is still a proper subclass of Kirchberg algebras. In this context S\o rensen \cite{Sor} give a combinatorial classification result for simple graph $C^*$-algebras of graphs with finitely many vertices and countably many edges. This result and R\o rdam's one combine to give a combinatorial classification result for unital simple graph $C^*$-algebras.

Finally, Katsura \cite{Kat1} constructed a class of algebras $\mathcal{O}_{A,B}$ (where $A\in M_n(\Z^+)$, $B\in M_n(\Z)$) of combinatorial nature, generalizing Cuntz-Krieger algebras. Moreover, by using the Kirchberg-Phillips Theorem, he showed that any Kirchberg algebra can be represented (up to isomorphism) by an algebra of this family. Despite  its purely combinatorial nature, Katsura analyzed these algebras via the notion of topological graph \cite{Kat2}. So, in the study of its structure and properties, the purely algebraic combinatorial framework is somehow forgotten. But these algebras clearly appear as a suitable combinatorial model for Kirchberg algebras, opening the possibility of an algebraic study of this relevant class of algebras.\vspace{.2truecm}

Very recently, the first named author\cite{Exel1} developed a complete theory for studying $C^*$-algebras from a combinatorial point of view. In its early stages, his theory produced interesting approaches to Cuntz-Krieger algebras for infinite matrices \cite{ExelLaca} and more general objects \cite{ELQ}. Concretely, he analizes how to represent large families of algebras via standarized constructions of combinatorial groupoids, and how to express these algebras as  partial crossed products of commutative $C^*$-algebras by inverse semigroups and groups. In this way, he faithfully reflects the symbolic dynamic information associated to these algebras, and so opens the possibility of studying this kind of algebras by its intrinsic combinatorial structure, specially for its classification by combinatorial methods.\vspace{.2truecm}

In this paper, we will study Katsura algebras from a combinatorial point of view, using the techniques developed by the first named author. At first sight, the more natural approach would be to mimic Exel and Laca's strategy \cite{ExelLaca}, by searching a suitable group acting by partial automorphisms on a commutative $C^*$-algebra; in order to identify such a commutative algebra, the techniques developed in \cite{ELQ} would be essential. But the problem is that the natural group $G$ associated to such a picture of $\OAB$ does not enjoy the right properties, so that the techniques of \cite{ELQ} do not apply to give us the desired description of $\OAB$ as a partial group crossed product. So, we need to start from a more basic point of view, by constructing a semigroupoid $\Lambda$ whose properties allow us to use the techniques developed in \cite{Exel1} to present $\OAB$ as the $C^*$-algebra of $\Lambda$. By proving that $\Lambda$ enjoys the right properties, we will then be able to present $\OAB$ as a groupoid C*-algebra and as an inverse semigroup crossed product. The aim is to use this combinatorial point of view to associate a symbolic dynamical model to $\OAB$, and then to use it to state its basic properties in a more natural way. We are specially interested in stating the nuclearity of these algebras, and characterizing simplicity and pure infiniteness using this approach of the problem, more natural and intuitive. The future goal to be attended will be to understand the symbolic dynamics model associated to $\OAB$, in the search of a combinatorial, purely algebraic version of Kirchberg-Phillips Theorem.\vspace{.2truecm}

The contents of this paper can be summarized as follows. In Section 2, we recall the definitions and results about Katsura algebras $\OAB$ we will use in the paper. Section 3 is devoted to constructing a suitable semigroupoid $\LAB$ associated to $\OAB$, to analyze its basic properties and to define a inverse semigroup $\SAB$ associated to it. In Section 4 we state a tight representation of $\LAB$ in $\OAB$, and we conclude that $\OAB$ is isomorphic to the full $C^*$-algebra of the tight groupoid $\mathcal{G}_{\LAB}$ associated to $\SAB$; also, the amenability of the groupoid is proved. In Section 5 we compute an explicit form of $\mathcal{G}_{\LAB}^{(0)}$, and thus we obtain an explicit, operational presentation of the groupoid in terms of the action of the inverse semigroup $\mathcal{S}^{A,B}$ generated by the partial isometries of $\OAB$. By using this picture, we characterize when the groupoid is minimal and essentially free in Sections 6 and 7 respectively. In Section 8, under mild hypotheses, we obtain a characterization of simplicity for $\OAB$ improving Katsura's results, who only gives sufficient conditions for simplicity. In Section 9 we state sufficient conditions for the algebra $\OAB$ be purely infinite (simple). In the simple case, because of the characterization result obtained before, it extends the case stated by Katsura to more general situations.

\section{Summary on Katsura algebras}

In this section we will quickly recall the definition and basic properties of Katsura algebras that will be needed in the sequel. All the contents of this section are borrowed from \cite{Kat1}.

\begin{definition}\label{Def:KatAlgData}
{\rm Let $N\in \Ninf$, let $A\in M_N(\Zplus)$ and $B\in M_N(\Z)$ be row-finite matrices. Define a set $\OmA$ by
$$\OmA:=\{ (i,j)\in \{ 1, 2, \dots ,N\}\times\{ 1, 2, \dots ,N\} \mid A_{i,j}\geq 1 \}.$$
For each $i\in \{ 1, 2, \dots ,N\}$, define a set $\OmA(i)\subset \{ 1, 2, \dots ,N\}$ by
$$\OmA(i):=\{ j \in \{ 1, 2, \dots ,N\}\mid (i,j)\in \OmA\}.$$
Notice that, by definition, $\OmA(i)$ is finite for all $i$.
Finally, fix the following relation:
\begin{quotation}
(0) $\OmA(i)\ne \emptyset$ for all $i$, and $B_{i,j}=0$ for $(i,j)\not\in \OmA$.
\end{quotation}}
\end{definition}

\begin{remark}\label{Rem:Row_Finite}
{\rm Notice that the first statement on Condition (0) implies that $A$ has no identically zero rows.
}
\end{remark}

With these data we can define the algebras

\begin{definition}\label{Def:KatAlgAlgebra}
{\rm With the data of Definition \ref{Def:KatAlgData}, define $\OAB$ to be the universal $C^*$-algebra generated by mutually orthogonal projections $\{ q_i\}_{i=1}^N$, partial unitaries $\{ u_i\}_{i=1}^N$ with $u_iu_i^*=u_i^*u_i=q_i$, and partial isometries $\{ s_{i,j,n}\}_{(i,j)\in \OmA, n\in \Z}$ satisfying the relations:
\begin{enumerate}
\item[(i)] $s_{i,j,n}u_j=s_{i,j, n+A_{i,j}}$ and $u_is_{i,j,n}=s_{i,j, n+B_{i,j}}$ for all $(i,j)\in \OmA$ and $n\in \Z$.
\item[(ii)] $s_{i,j,n}^*s_{i,j,n}=q_j$ for all $(i,j)\in \OmA$ and $n\in \Z$.
\item[(iii)] $q_i=\sum\limits_{j\in \OmA(i)}\sum\limits_{n=1}^{A_{i,j}}s_{i,j,n}s_{i,j,n}^*$ for all $i$.
\end{enumerate}}
\noindent We will denote $p_{i,j,n}:=s_{i,j,n}s_{i,j,n}^*$.
\end{definition}

\begin{noname}\label{KatsuraConditions}
{\rm Now, the following facts holds:
\begin{enumerate}
\item The $C^*$-algebra $\OAB$ is separable, nuclear and in the UCT class \cite[Proposition 2.9]{Kat1}.
\item If the matrices $A,B$ satisfy the following addition properties:
\begin{enumerate}
\item $A$ is irreducible.
\item $A_{i,i}\geq 2$ and $B_{i,i}=1$ for every $1\leq i\leq N$.
\end{enumerate}
then the $C^*$-algebra $\OAB$ is simple and purely infinite, and hence a Kirchberg algebra \cite[Proposition 2.10]{Kat1}.
\item The $K$-groups of $\OAB$ are \cite[Proposition 2.6]{Kat1}:
\begin{enumerate}
\item $K_0(\OAB)\cong \mbox{coker}(I-A)\oplus \mbox{ker}(I-B)$,
\item $K_1(\OAB)\cong \mbox{coker}(I-B)\oplus \mbox{ker}(I-A)$.
\end{enumerate}
\item Every Kirchberg algebra can be represented, up to isomorphism, by an algebra $\OAB$ for matrices $A,B$ satisfying the conditions $(2)(a\& b)$ \cite[Proposition 4.5]{Kat2}.
\end{enumerate}
}
\end{noname}

Besides  that, it is useful to recall the following facts

\begin{lemma}\label{Lem:NormalFormOAB} Denote by $I$ the multiindex set $\{ (i_1,i_2,n_1)(i_2,i_3,n_2)\cdots (i_k,i_{k+1},n_k)\}$, where $(i_j, i_{j+1})\in \OmA$ and $n_j\in \Z$ for all $j$, and define $s(I)=i_1$, $t(I)=i_{k+1}$. Given a multiindex $I$ as above, define $s_I:= s_{i_1,i_2,n_1}s_{i_2,i_3,n_2}\cdots s_{i_k,i_{k+1},n_k}$. Then:
\begin{enumerate}
\item $\OAB=\overline{\text{span}}\{ s_Is_J^*\mid I, J \mbox{ multiindices with }t(I)=t(J)\}$.
\item For any multiindex $I$ there exists a unique multiindex $\widehat{I}$ such that $1\leq \widehat{n}_j\leq A_{i_j, i_{j+1}}$ for all $j$ and a unique $t_I\in \Z$ such that $s_I=s_{\widehat{I}}u_{k+1}^{t_I}$.
\item If $T^{A,B}=\{ s_Iu_{r(I)}^t \mid I \text{ multiindex in (2)}, t\in \Z\}$ then the multiplicative sub-semigroup of $\OAB$, denote by  $\mathcal{S}^{A,B}$ generated by $T^{A,B}\cup ({T^{A,B}})^*$ is an inverse semigroup with zero.
\end{enumerate}
\end{lemma}

In the case that $B=(0)$ \cite[Proposition 4.7]{Kat1}, $\mathcal{O}_{A, (0)}$ is isomorphic to the Cuntz-Krieger algebra $\mathcal{O}_A$ (the Exel-Laca algebra \cite{ExelLaca} if $N=\infty$) generated by mutually orthogonal projections $\{ q_i\}_{i=1}^N$,  and partial isometries $\{ s_{i,j,n} \mid {(i,j)\in \OmA, 1\leq n\leq A_{i,j}}\}$ satisfying the relations:
\begin{enumerate}
\item[(i)] $s_{i,j,n}^*s_{i,j,n}=q_j$ for all $(i,j)\in \OmA$ and $1\leq n\leq A_{i,j}$.
\item[(ii)] $q_i=\sum\limits_{j\in \OmA(i)}\sum\limits_{n=1}^{A_{i,j}}p_{i,j,n}$ for all $i$.
\end{enumerate}
where $p_{i,j,n}:=s_{i,j,n}s_{i,j,n}^*$.

\begin{noname}\label{CuntzSemigroup}
{\rm Notice that for any matrix $B$ there is a natural injective $\ast$-homomorphism $\mathcal{O}_A\hookrightarrow \OAB$ \cite[Proposition 4.5]{Kat1}. When the multiindex sets $\{ (i_1,i_2,n_1)(i_2,i_3,n_2)\cdots (i_k,i_{k+1},n_k)\}$ restrict to the case where $(i_j, i_{j+1})\in \OmA$ and $1\leq n_j\leq A_{i_j,i_{j+1}}$ for all $j$, then $\mathcal{O}_A=\overline{\text{span}}\{ s_Is_J^*\mid I, J \mbox{ multiindices with }t(I)=t(J)\}$. Moreover, if $T^{A}=\{ s_I \mid I \text{ multiindex}\}$ then the semigroup $\mathcal{S}^{A}$ generated by $T^{A}\cup {T^{A}}^*$ is an inverse semigroup with zero, and the semilattices of projections $E(\mathcal{S}^{A})$ and $E(\mathcal{S}^{A,B})$ coincide.
}
\end{noname}

\section{The semigroupoid $\LAB$}

In this section we will associate to any pair of matrices $A,B$ satisfying Definition \ref{Def:KatAlgData} a semigroupoid $\LAB$, and we will state the essential properties enjoyed by this semigroupoid. Also, we will briefly describe the inverse semigroup $\SAB$ associated to $\LAB$ according to \cite[Section 14]{Exel1}. The essential idea to be kept in mind is that our semigroupoid is a sort of path semigroupoid.

\begin{definition}\label{Def:semigroupoidAB}
{\rm For matrices $A,B$ satisfying Definition \ref{Def:KatAlgData}, fix a set
$$G:=\{ h_i\mid 1\leq i\leq N\}\cup \{ g_{i,j,n}\mid (i,j)\in \OmA, n\in \Z\}.$$
Then, we define $\LAB$ to be the semigroupoid generated by $G$, satisfying:
\begin{enumerate}
\item The (partial) operation is induced by the admissible pairs
$$\{ (h_i, g_{i,j,n}), (g_{i,j,n}, h_j), (g_{i,j,n}, g_{j,k,m}), (h_i,h_i)\}$$
for all $1\leq i,j\leq N$, $(i,j), (j,k)\in \OmA$, $n,m\in \Z$.
\item $g_{i,j,n}h_j=g_{i,j, n+A_{i,j}}$ and $h_ig_{i,j,n}=g_{i,j, n+B_{i,j}}$ for all $(i,j)\in \OmA$ and $n\in \Z$.
 \end{enumerate}
}
\end{definition}

\begin{remark}\label{Rem:LABwelldefined}
{\rm Notice that, by the properties in Definition \ref{Def:semigroupoidAB}, the associative laws of the partial operation (see e.g. \cite[Section 14]{Exel1}) hold, so that $\LAB$ is a well-defined semigroupoid.
}
\end{remark}

Now, we will show that there exists a standard form for the elements of $\LAB$, and that this form is unique.

\begin{lemma}\label{Lem:Tecnic1LAB}
$g_{i,j,n}g_{j,k,m}=g_{i,j,n-A_{i,j}}g_{j,k,m+B_{j,k}}=g_{i,j,n+A_{i,j}}g_{j,k,m-B_{j,k}}$
\end{lemma}
\begin{proof}
Simply notice that
$$g_{i,j,n-A_{i,j}}g_{j,k,m+B_{j,k}}=g_{i,j,n-A_{i,j}}h_jg_{j,k,m}=g_{i,j,n}g_{j,k,m}$$
and
$$g_{i,j,n+A_{i,j}}g_{j,k,m-B_{j,k}}=g_{i,j,n}h_jg_{j,k,m-B_{j,k}}=g_{i,j,n}g_{j,k,m}.$$
\end{proof}

By Definition \ref{Def:semigroupoidAB} and Lemma \ref{Lem:Tecnic1LAB}, it is easy to conclude that

\begin{lemma}\label{Lem:InitialFormLAB}
The elements of $\LAB$ have one of the following two forms:
\begin{enumerate}
\item $h_i^t$ for $1\leq i\leq N$ and $t\in\Z$.
\item $g_{i_1,i_2,n_1}g_{i_2,i_3,n_2}\cdots g_{i_k,i_{k+1},n_k}$ for $(i_j, i_{j+1})\in \OmA$, $n_j\in \Z$ for all $j$.
\end{enumerate}
\end{lemma}

Moreover, we can prove the following result.

\begin{lemma}\label{Lem:NormalFormLAB}
For any element $\alpha=g_{i_1,i_2,n_1}g_{i_2,i_3,n_2}\cdots g_{i_k,i_{k+1},n_k}$ with $(i_j, i_{j+1})\in \OmA$, $n_j\in \Z$, there exists $1\leq m_j\leq A_{i_j, i_{j+1}}$ for $1\leq j\leq k-1$ and $m_k\in \Z$ such that
$$\alpha=g_{i_1,i_2,m_1}g_{i_2,i_3,m_2}\cdots g_{i_{k-1},i_{k},m_{k-1}}g_{i_k,i_{k+1},m_k}.$$
\end{lemma}
\begin{proof}
We will prove the result by induction on $k$. The case $k=2$ being clear, let us suppose that the result is true for $l< k$, and let us prove it for $l=k$. Then, suppose that
$$\alpha=g_{i_1,i_2,n_1}g_{i_2,i_3,n_2}\cdots g_{i_{k-1},i_{k},n_{k-1}}g_{i_k,i_{k+1},n_k}$$
with $(i_j, i_{j+1})\in \OmA$, $n_j\in \Z$. By induction hypothesis, there exists $1\leq m_j\leq A_{i_j, i_{j+1}}$ for $1\leq j\leq k-2$ and $m'_{k-1}\in \Z$ such that
$$\alpha=g_{i_1,i_2,m_1}g_{i_2,i_3,m_2}\cdots g_{i_{k-2},i_{k-1},m_{k-2}}g_{i_{k-1},i_{k},m'_{k-1}}g_{i_k,i_{k+1},n_k}.$$

Now, if $m'_{k-1}\geq 0$, there exist (unique) $1\leq m_{k-1}\leq A_{i_{k-1}, i_{k}}$ and $t\in \Zplus$ such that $m'_{k-1}=m_{k-1}+tA_{i_{k-1}, i_{k}}$. Hence, by Lemma \ref{Lem:Tecnic1LAB}, $g_{i_{k-1},i_{k},m'_{k-1}}g_{i_k ,i_{k+1},n_k}=g_{i_{k-1},i_{k},m_{k-1}}g_{i_k,i_{k+1},n_k+t B_{i_k, i_{k+1}}}$, so that the result holds by defining $m_k=n_k+t B_{i_k, i_{k+1}}$.

On the other hand, if $m'_{k-1}< 0$, there exist (unique) $1\leq m_{k-1}\leq A_{i_{k-1}, i_{k}}$ and $t\in \Z$ a negative integer such that $m'_{k-1}=m_{k-1}+tA_{i_{k-1}, i_{k}}$. Then, by Lemma \ref{Lem:Tecnic1LAB},
$$g_{i_{k-1},i_{k},m'_{k-1}}g_{i_k ,i_{k+1},n_k}= g_{i_{k-1},i_{k},m'_{k-1}}g_{i_k ,i_{k+1},n_k+ \vert t\vert B_{i_k, i_{k+1}}- \vert t\vert B_{i_k, i_{k+1}}}=$$
$$g_{i_{k-1},i_{k},m_{k-1}+tA_{i_{k-1}, i_{k}}}h_{i_k}^{\vert t\vert}g_{i_k ,i_{k+1},n_k- \vert t\vert B_{i_k, i_{k+1}}}=g_{i_{k-1},i_{k},m_{k-1}}g_{i_k ,i_{k+1},n_k- \vert t\vert B_{i_k, i_{k+1}}},$$
so that the result holds by defining $m_k=n_k-\vert t\vert B_{i_k, i_{k+1}}=n_k+t B_{i_k, i_{k+1}}$.
\end{proof}

We will call this special form of the elements of $\LAB$ the \emph{standard form}.

\begin{lemma}\label{Lem:UniqueNormalFormLAB}
The standard form of an element of $\LAB$ is unique.
\end{lemma}
\begin{proof}
First, notice that the length of an element of $\LAB$ (i.e. the number of $g_{i,j,n}$'s appearing in its expression), as well as the $(i_j, i_{j+1})$'s involved, do not change by the defining relations. Thus, what we have to prove is that, whenever $\alpha \in \LAB$ can be expressed both as
$$\alpha=g_{i_1,i_2,n_1}g_{i_2,i_3,n_2}\cdots g_{i_{k-1},i_{k},n_{k-1}}g_{i_k,i_{k+1},n_k}$$
and as
$$\alpha=g_{i_1,i_2,m_1}g_{i_2,i_3,m_2}\cdots g_{i_{k-1},i_{k},m_{k-1}}g_{i_k,i_{k+1},m_k}$$
with $1\leq n_j ,m_j\leq A_{i_j, i_{j+1}}$ for $1\leq j\leq k-1$ and $n_k, m_k\in \Z$, then $n_j=m_j$ for all $j$.

By Lemma \ref{Lem:Tecnic1LAB}, there exist $r_1, \dots, r_{k-1}, s_1, \dots , s_{k-1}\in \Z$ such that:
\begin{enumerate}
\item $n_1+r_1A_{i_1, i_2}=m_1+s_1A_{i_1, i_2}$,
\item $n_j+r_jA_{i_j, i_{j+1}}-r_{j-1}B_{i_j, i_{j+1}}=m_j+s_jA_{i_j, i_{j+1}}-s_{j-1}B_{i_j, i_{j+1}}$ for all $2\leq j\leq k-1$, and
\item $n_k-r_{k-1}B_{i_k, i_{k+1}}=m_k-s_{k-1}B_{i_k, i_{k+1}}$.
\end{enumerate}
Now, notice that $n_1+r_1A_{i_1, i_2}=m_1+s_1A_{i_1, i_2}$ implies $n_1-m_1=(s_1-r_1)A_{i_1, i_2}$. Since $1\leq n_1,m_1\leq A_{i_1, i_2}$, we get that $s_1=r_1$ and thus $n_1=m_1$.

Next, $n_2+r_2A_{i_2, i_{3}}-r_{1}B_{i_2, i_{3}}=m_2+s_2A_{i_2, i_{3}}-s_{1}B_{i_2, i_{3}}$, that by the previous computation implies $n_2-m_2=(s_2-r_2)A_{i_2, i_3}$. Again, since $1\leq n_2,m_2\leq A_{i_2, i_3}$, we get that $s_2=r_2$ and thus $n_2=m_2$.

Recurrence on this argument shows that $s_j=r_j$ and $n_j=m_j$ for all $j$, as desired.
\end{proof}

Recall from \cite[Section 14]{Exel1} that given any semigroupoid $\Lambda$, we define $\widetilde{\Lambda}=\Lambda\cup \{ 1\}$, and that an element $f\in \widetilde{\Lambda}$ is said to be monic if for every $g,h\in \widetilde{\Lambda}$ we have that $fg=fh$ implies $g=h$. Symmetrically, an element $f\in \widetilde{\Lambda}$ is said to be epic if for every $g,h\in \widetilde{\Lambda}$ we have that $gf=hf$ implies $g=h$.

\begin{proposition}\label{Prop:LAB-Monic}
Every element of $\LAB$ is monic.
\end{proposition}
\begin{proof}
We will write all the elements in standard form. Then, we need to distinguish various cases:\vspace{.2truecm}

(1) $f=h_i^t, g=h_j^r, h=h_k^s$. Then, the definition of $\LAB$ forces $i=j=k$, and since $fg=h_i^{t+r}=h_i^{t+s}=fh$, we conclude that $r=s$, i.e. $g=h$.\vspace{.2truecm}

(2) $f=h_{i_1}^t$,
$$g=g_{i_1,i_2,n_1}g_{i_2,i_3,n_2}\cdots g_{i_{k-1},i_{k},n_{k-1}}g_{i_k,i_{k+1},n_k}$$
and
$$h=g_{i_1,i_2,m_1}g_{i_2,i_3,m_2}\cdots g_{i_{k-1},i_{k},m_{k-1}}g_{i_k,i_{k+1},m_k}.$$
Then,
$$fg=fh=g_{i_1,i_2,p_1}g_{i_2,i_3,p_2}\cdots g_{i_{k-1},i_{k},p_{k-1}}g_{i_k,i_{k+1},p_k}$$
in standard form, which means that there exist $l_2, \dots ,l_{k},t_2, \dots ,t_{k}\in \Z$ such that:
\begin{enumerate}
\item[(i)] $p_1=n_1+lB_{i_1,i_2}-l_2A_{i_1,i_2}=m_1+lB_{i_1,i_2}-t_2A_{i_1,i_2}$,
\item[(ii)] $p_j=n_j+l_jB_{i_j,i_{j+1}}-l_{j+1}A_{i_j,i_{j+1}}=m_j+t_jB_{i_j,i_{j+1}}-t_{j+1}A_{i_j,i_{j+1}}$ for all $2\leq j\leq k-1$, and
\item[(iii)] $p_k=n_k+l_kB_{i_k,i_{k+1}}=m_k+t_kB_{i_k,i_{k+1}}$.
\end{enumerate}
The same argument as in the proof of Lemma \ref{Lem:UniqueNormalFormLAB} shows that $l_j=t_j$ and thus $n_j=m_j$ for all $j$, so that $g=h$.\vspace{.2truecm}

(3) $f=g_{i_1,i_2,n_1}g_{i_2,i_3,n_2}\cdots g_{i_{k-1},i_{k},n_{k-1}}g_{i_k,i_{k+1},n_k}$ in standard form, while $g=h_i^r, h=h_j^s$. Thus, we have $i=j=i_{k+1}$ and $n_k+r=n_k+s$, so that $r=s$ and thus $g=h$.\vspace{.2truecm}

(4) $f=g_{j_1,j_2,n_1}g_{j_2,j_3,n_2}\cdots g_{j_{l-1},j_{l},n_{l-1}}g_{j_l,j_{l+1},n_l}$, while $$g=g_{i_1,i_2,m_1}g_{i_2,i_3,m_2}\cdots g_{i_{k-1},i_{k},m_{k-1}}g_{i_k,i_{k+1},m_k}$$
and
$$h=g_{i_1,i_2,p_1}g_{i_2,i_3,p_2}\cdots g_{i_{k-1},i_{k},p_{k-1}}g_{i_k,i_{k+1},p_k}$$
in standard form. By Lemma \ref{Lem:UniqueNormalFormLAB} there is no loss of generality in assuming that $l=1$, and that $f=g_{i_0,i_1,n_0}$ with $n_0\in \Z$. Now,
$$fg=fh=g_{i_0,i_1,q_0}g_{i_1,i_2,q_1}g_{i_2,i_3,q_2}\cdots g_{i_{k-1},i_{k},q_{k-1}}g_{i_k,i_{k+1},q_k}$$
in normal form, so that there exist $l_1, \dots ,l_{k},t_1, \dots ,t_{k}\in \Z$ such that:
\begin{enumerate}
\item[(i)] $q_0=n_0-l_1A_{i_0,i_1}=n_0-t_1A_{i_0,i_1}$,
\item[(ii)] $q_j=m_j+l_jB_{i_j,i_{j+1}}-l_{j+1}A_{i_j,i_{j+1}}=p_j+t_jB_{i_j,i_{j+1}}-t_{j+1}A_{i_j,i_{j+1}}$ for all $1\leq j\leq k-1$, and
\item[(iii)] $q_k=m_k+l_kB_{i_k,i_{k+1}}=p_k+t_kB_{i_k,i_{k+1}}$.
\end{enumerate}
Again the same argument as in the proof of Lemma \ref{Lem:UniqueNormalFormLAB} shows that $l_j=t_j$ and thus $m_j=p_j$ for all $j$, so that $g=h$.
\end{proof}

By an analogue argument, we can show the characterization of epic elements in $\LAB$. For this we need a definition.

\begin{definition}\label{Def:Cond(E)}
{\rm We say that the matrix $B$ satisfies Condition (E) when $B_{i,j}=0$ if and only if $(i,j)\not\in \OmA$.
}
\end{definition}

\begin{lemma}\label{Lem:LAB-Epic}
Every element of $\LAB$ is epic if and only if $B$ satisfies Condition (E).
\end{lemma}

\begin{noname}
{\rm Given two elements $f,g$ in a semigroupoid $\Lambda$, recall that:
\begin{enumerate}
\item $f$ divides $g$ (or $g$ is a multiple of $f$), denoted $f\vert g$, if either $f=g$ or there exists $h\in \Lambda$ such that $fh=g$.
\item $f$ and $g$ intersect if they admit a common multiple, i.e. there exists $m\in \Lambda$ such that both $f\vert m$ and $g\vert m$. We denote this relation by $f\Cap g$.
\item $f$ and $g$ are disjoint if they do not intersect, and we denote it by $f\perp g$.
\item If $f,g\in \Lambda$ and $f\Cap g$, we say that $m\in \Lambda$ is a least common multiple if it is a common multiple, and whenever $h\in \Lambda$ is a common multiple of $f$ and $g$, then $m\vert h$. If it is unique, we will denote it by $m=\text{lcm}(f,g)$.
\end{enumerate}}
\end{noname}

Then, we have the following result.

\begin{proposition}\label{Prop:existsLCMinLAB}
There exists a unique least common multiple for any pair of intersecting elements of $\LAB$.
\end{proposition}
\begin{proof}
We will assume that $f,g\in \Lambda$ and $f\Cap g$, and we will prove the existence of a least common multiple for that pair. We need to distinguish various cases:\vspace{.2truecm}

(1) $f=h_i^t, g=h_j^r$. Then, if $m$ is a common multiple, since either there exist $f', g'\in \LAB$ such that $ff'=gg'$ or $m=f$ or $m=g$, the definition of $\LAB$ forces $i=j$. Conversely, if $i=j$, and we assume $t< r$  (when $t=r$, then $f=g$ and there is nothing to prove), then $fh_i^{r-t}=g$.\vspace{.2truecm}

(2) $f=h_{j}^t$ and $g=g_{i_1,i_2,n_1}g_{i_2,i_3,n_2}\cdots g_{i_{k-1},i_{k},n_{k-1}}g_{i_k,i_{k+1},n_k}$. As in case (1), since $f\Cap g$, we have that $j=i_1$. Conversely, if $j=i_1$, then
$$fg_{i_1,i_2,n_1-tB_{i_1,i_2}}g_{i_2,i_3,n_2}\cdots g_{i_{k-1},i_{k},n_{k-1}}g_{i_k,i_{k+1},n_k}=g.$$\vspace{.2truecm}

Notice that the argument in (2) shows that:
\begin{enumerate}
\item[(i)] In case (1) $m=h_i^{r}$ must be a least common multiple for $f$ and $g$.
\item[(ii)] In case (2) $m=g$ must be a least common multiple for $f$ and $g$.
\end{enumerate}\vspace{.2truecm}

(3) $f=g_{i_1,i_2,n_1}g_{i_2,i_3,n_2}\cdots g_{i_{k-1},i_{k},n_{k-1}}g_{i_k,i_{k+1},n_k}$ and
$g=g_{j_1,j_2,m_1}g_{j_2,j_3,m_2}\cdots g_{j_{l-1},j_{l},m_{l-1}}g_{j_l,j_{l+1},m_l}$, written in standard form. Since $f\Cap g$, there exist $f', g'\in \LAB$ such that $ff'=gg'$. Hence, if $k\leq l$ then we have that $j_t=i_t$ for all $1\leq t\leq k$ and $n_t=m_t$ for all $1\leq t\leq k-1$. Moreover, define $\widehat{f}=g_{i_1,i_2,n_1}g_{i_2,i_3,n_2}\cdots g_{i_{k-1},i_{k},n_{k-1}}$, so that $f=\widehat{f}g_{i_k,i_{k+1},n_k}$ and $g=\widehat{f}g_{i_k,i_{k+1},m_k}g_{i_{k+1},i_{k+2},m_{k+1}}g_{i_{k+2},i_{k+3},m_{k+2}}\cdots g_{i_l,i_{l+1},m_l}$. Since
$$\widehat{f}g_{i_k,i_{k+1},n_k}f'=\widehat{f}g_{i_k,i_{k+1},m_k}g_{i_{k+1},i_{k+2},m_{k+1}}g_{i_{k+2},i_{k+3},m_{k+2}}\cdots g_{i_l,i_{l+1},m_l}g',$$
Lemma \ref{Lem:Tecnic1LAB} implies that $n_k-m_k=tA_{i_k,i_{k+1}}$ for some $t\in \Z$. Conversely, if these restrictions hold, then:
\begin{enumerate}
\item[(i)] If $k<l$, then
$$fg_{i_{k+1},i_{k+2},m_{k+1}- t B_{i_{k+1},i_{k+2}}}g_{i_{k+2},i_{k+3},m_{k+2}}\cdots g_{i_l,i_{l+1},m_l}=g,$$
so that $f\Cap g$ and $g$ is a least common multiple for that pair.
\item[(ii)] If $k=l$ and $f\ne g$, since $n_k=m_k+tA_{i_k,i_{k+1}}$  for some $t\in \Z$, we can assume that $t> 0$ (the other case is symmetric), and thus
$$gh_{i_{k+1}}^t= \widehat{f}g_{i_k,i_{k+1},m_k}h_{i_{k+1}}^t=\widehat{f}g_{i_k,i_{k+1},n_k}=f,$$
so that $f$ is a least common multiple for the pair.
\end{enumerate}
Uniqueness is a consequence of Proposition \ref{Prop:LAB-Monic} (see \cite[14.8]{Exel1}), so we are done.
\end{proof}

\begin{remark}\label{Rem:LAB-StandingHypothesis}
{\rm Notice that $\LAB$ satisfies the Standing Hypothesis \cite[14.7]{Exel1} because of Proposition \ref{Prop:LAB-Monic} and Proposition \ref{Prop:existsLCMinLAB}
}
\end{remark}

Given an element $f$ in a semigroupoid $\Lambda$, we define $\Lambda^f=\{ g\in \Lambda \mid (f,g)\in \Lambda^{(2)}\}$. We say that an element is a spring if $\Lambda^f=\emptyset$. We say that $\Lambda$ is categorical if given $f,g\in \Lambda$ then either $\Lambda^f=\Lambda^g$ or $\Lambda^f\cap\Lambda^g=\emptyset$. Then, we have the following result.

\begin{lemma}\label{Lem:LAB-CategoricalNoSprings}
The semigroupoid $\LAB$ has no springs and is categorical.
\end{lemma}
\begin{proof}
Given any $f\in \LAB$, $f$ is either $h_j^t$ of $\widehat{f}g_{i,j,n}$. Because of Definition \ref{Def:semigroupoidAB}, in  both cases $\LAB^f=\LAB^{h_j}$. Then, $h_j\in \LAB^f$, so that $\LAB$ has no springs. The property of being categorical is then obvious.
\end{proof}

Given a semigroupoid $\Lambda$, and $\Gamma \subseteq \Lambda$ a subset, we say that a subset $H\subset \Gamma$ is a cover for $\Gamma$ if for every $f\in \Gamma$ there exists $h\in H$ such that $f\Cap h$. If moreover the elements of $H$ are mutually disjoint then $H$ is called a partition \cite[Definition 15.3]{Exel1}.

\begin{proposition}\label{Prop:LAB-ExistsFinitePartition}
For every $f\in \LAB$ and for every $h\in \LAB^f$ there exists a finite partition $H\subseteq \LAB^f$ such that $h\in H$.
\end{proposition}
\begin{proof}
As noticed in the proof of Lemma \ref{Lem:LAB-CategoricalNoSprings}, we can assume that $f=h_j$ for some $1\leq j\leq N$. We will distinguish two cases:\vspace{.2truecm}

(1) $h=h_j^t$ with $t\in\N$. Then, by the arguments in the proof of Proposition \ref{Prop:existsLCMinLAB}, every $g\in \LAB^{h_j}$ intersects with $h_i$, so that it is enough to fix $H:=\{ h_j\}$.\vspace{.2truecm}

(2) $h=g_{i_1,i_2,m_1}g_{i_2,i_3,m_2}\cdots g_{i_{k-1},i_{k},m_{k-1}}g_{i_k,i_{k+1},m_k}$ (written in standard form) with $i_1=j$. Then, $1\leq m_j\leq A_{i_j, i_{j+1}}$ for $1\leq j\leq k-1$ and $m_k\in \Z$, so that there exists exactly one $t\in \Z$ such that $m_k\in [tA_{i_k,i_{k+1}}+1, (t+1)A_{i_k,i_{k+1}}]$ (i.e., $m_k=n_k+tA_{i_k,i_{k+1}}$ for a unique $1\leq n_k\leq A_{i_k,i_{k+1}}$ and a unique $t\in \Z$). Now, we will construct the finite partition $H$ using a ``rooted tree'' with ``root'' $i_1$ and paths constructed using as ``edges'' the pairs $(i_j, i_{j+1})$ appearing in the expression of $h$; the idea is to use a minimal expansion of the root containing the path associated to $h$, and then complete this expansion in each vertex adding as many generators as $A_{i,j}$. Notice that, by hypothesis, fixed any $i_j$ ($1\leq j\leq k$), the set $\OmA(i_j)$ is finite, and thus so is the set of elements $p\in \N$ such that $1\leq p\leq A_{i_j,l}$ for any $l\in \OmA(i_j)$. Now, we define $H$ by recurrence, as follows:
\begin{enumerate}
\item[(i)] Define $H_1:=\{ g_{i_1, l, m} \mid l\in \OmA(i_1), 1\leq l\leq A_{i_1,l}\}\setminus \{  g_{i_1, i_2, m_1}\}$. Notice that, by the argument in the proof of Proposition \ref{Prop:existsLCMinLAB}, the elements of $H_1$ are mutually disjoint. Moreover, every element of $\LAB^{h_{i_1}}$, except these elements starting in $ g_{i_1, i_2, m_1}$, intersects with exactly one element of $H_1$. Also, $H_1$ is finite.
\item[(ii)] Define $H_2:=H_1\cup (\{ g_{i_1, i_2, m_1}g_{i_2, l, m} \mid l\in \OmA(i_2), 1\leq m\leq A_{i_2,l}\}\setminus \{  g_{i_1, i_2, m_1}g_{i_2, i_3, m_2}\})$. As in  case (i), the elements of $H_2$ are mutually disjoint, and every element of $\LAB^{h_{i_1}}$, except these elements starting in $ g_{i_1, i_2, m_1}g_{i_2, i_3, m_2}$, intersects with exactly one element of $H_2$. Also, $H_2$ is finite.
\item[(iii)] Suppose defined $H_j$, $2\leq j<k-1$, and define $H_{j+1}:=H_j\cup (\{ g_{i_1, i_2, m_1}g_{i_2, i_3, m_2} \dots g_{i_j,l,m}\mid l\in \OmA(i_j), 1\leq m\leq A_{i_j,l}\}\setminus \{  g_{i_1, i_2, m_1}g_{i_2, i_3, m_2} \dots g_{i_j,i_{j+1},m_j}\})$. As in the case (ii), the elements of $H_{j+1}$ are mutually disjoint, and every element of $\LAB^{h_{i_1}}$, except these elements starting in $ g_{i_1, i_2, m_1}g_{i_2, i_3, m_2}\dots g_{i_j,i_{j+1},m_j}$, intersects with exactly one element of $H_{j+1}$. Also, $H_{j+1}$ is finite.
\item[(iv)] Define $H:=H_{k-1}\cup (\{ g_{i_1, i_2, m_1}g_{i_2, i_3, m_2} \dots g_{i_k,l,n+tA_{i_k,l}}\mid l\in \OmA(i_j), 1\leq n\leq A_{i_k,l}\})$. As in  case (iii), the elements of $H$ are mutually disjoint, and every element of $\LAB^{h_{i_1}}$ intersects with exactly one element of $H$. Also, $H$ is finite, and $h\in H$.
\end{enumerate}
So we are done.
\end{proof}

\begin{noname}\label{SAB}
{\rm Now, we will describe the inverse semigroup $\SAB$ associated to $\LAB$. In order to define the elements of $\SAB$, by Lemma \ref{Lem:LAB-CategoricalNoSprings}, we have
$$\mathcal{Q}=\{ \LAB^f \mid f\in \LAB\}\cup \{ \emptyset\}=\{ \LAB^{h_i}\mid 1\leq i\leq N\}\cup \{ \emptyset\}.$$
hence,
$$\SAB=\{ (f, \LAB^f, g) \mid f,g\in \LAB\}\cup \{ (f, \LAB^f, 1) \mid f\in \LAB\}\cup
\{ (1, \LAB^g, g) \mid g\in \LAB\}$$
$$\cup \{ (1, \LAB^{h_i}, 1) \mid 1\leq i\leq N\}$$
with operation defined in \cite[Definition 14.15]{Exel1}. Notice that this semigroup is naturally endowed with an involution via the formula $(f,A,g)^*:=(g,A,f)$, and that it is an inverse semigroup with zero \cite[Theorem 14.16]{Exel1}. In fact, it is easy to see that $\SAB$ is generated (as semigroup) by the elements
$$\{(1, \LAB^{h_i}, h_i), (h_i, \LAB^{h_i}, 1) \mid 1\leq i\leq N\}$$
$$\cup \{ (g_{i,j,m}, \LAB^{h_j}, 1), (1, \LAB^{h_j}, g_{i,j,m})\mid (i,j)\in \OmA, m\in \Z\}.$$\vspace{.1truecm}

Also, the semilattice $E(\SAB)$ of projections of $\SAB$ is easily described: if we define $p_f:=(f, \LAB^f, f)$ and $q_{h_i}:= (1, \LAB^{h_i}, 1)$, then
$$E(\SAB)=\{ p_f\mid f\in \LAB\}\cup \{ q_i\mid 1\leq i\leq N\}.$$
Moreover, the following relations hold (see \cite[Proposition 19.4]{Exel1}):
\begin{enumerate}
\item[(i)] $p_fp_g=p_{\text{lcm}(f,g)}$ if $f\Cap g$.
\item[(ii)] $p_f\leq p_g$ if and only if $g\vert f$.
\item[(iii)] $p_f\perp p_g$ if $f\perp g$.
\item[(iv)] $p_f\leq q_{h_i}$ if $f\in \LAB^{h_i}$.
\item [(v)] $p_f\perp q_{h_i}$ if $f\not\in \LAB^{h_i}$.
\item[(vi)] $q_{h_i}\perp q_{h_j}$ if $i\ne j$.
\end{enumerate}
}
\end{noname}

\begin{remark}\label{Rem:Cond(E)}
{\rm  By \cite[Proposition 14.20]{Exel1}, if every element in $\LAB$ is epic then $\SAB$ is a $E^*$-unitary inverse semigroup, whence $\mathcal{G}_{\LAB}$ is Hausdorff by \cite[Proposition 6.2]{Exel1}.
}
\end{remark}

\section{Tight representations of $\LAB$}

We start this section recalling the definition of a tight representation of a semigroupoid $\Lambda$ in an inverse semigroup $\mathcal{S}$ with zero.

\begin{definition}[{\cite[Definition 15.1]{Exel1}}]\label{Def:SemigroupoidRep}
{\rm A representation of the semigroupoid $\Lambda$ in an inverse semigroup $\mathcal{S}$ with zero is a map $\pi:\Lambda \rightarrow \mathcal{S}$ such that, for every $f,g\in \Lambda$:
\begin{enumerate}
\item[(i)] $\pi_f\pi_g=\pi_{fg}$ if $(f,g)\in \Lambda^{(2)}$, and $\pi_f\pi_g=0$  if $f\perp g$. Moreover, if $q_f^{\pi}=\pi_f^*\pi_f$ and $p_g^{\pi}=\pi_g\pi_g^*$ are the associated projections, then
\item[(ii)] $p_f^{\pi}p_g^{\pi}=0$ if $f\perp g$.
\item[(iii)] $q_f^{\pi}p_g^{\pi}=p_g^{\pi}$ if $(f,g)\in \Lambda^{(2)}$.
\item[(iv)] $q_f^{\pi}p_g^{\pi}=0$ if $(f,g)\not\in \Lambda^{(2)}$.
\end{enumerate}
}
\end{definition}

If $\mathcal{S}$ is an inverse semigroup of partial isometries on a Hilbert space $\mathcal{H}$, we may consider the representation as a map $\pi:\Lambda \rightarrow \mathbb{B}(\mathcal{H})$ in which the projections $q_f^{\pi}=\pi_f^*\pi_f$ and $p_g^{\pi}=\pi_g\pi_g^*$ commute and the conditions (i-iv) are fulfilled.\vspace{.2truecm}

Given finite subsets $F,G\subset \Lambda$, we define $\Lambda^{F,G}:=\bigcap_{f\in F}\Lambda^f \cap \bigcap_{g\in G}\Lambda \setminus \Lambda^g$.

\begin{definition}[{\cite[Definition 15.1]{Exel1}}]\label{Def:TightRep}
{\rm A representation of the semigroupoid $\Lambda$ in a Hilbert space $\mathcal{H}$ is said to be tight if for every finite subsets $F,G\subset \Lambda$ and for every finite covering $H$ of $\Lambda^{F,G}$ we have that $\bigvee_{h\in H}p_h=q_{F,G}$, where $q_{F,G}:=\prod_{f\in F}q_f\prod_{g\in G}(1-q_g)$. In particular, if $H$ is a partition, then $\bigvee_{h\in H}p_h=\sum_{h\in H}p_h$.
}
\end{definition}

\begin{noname}\label{SpecialCase}
{\rm In the case of the semigroupoid $\LAB$, the sets $\LAB^{F,G}$ have a very precise description. Given $f\in \LAB$, denote $t(f)=j$ if either $f=h_j^t$ or $f=\widehat{f}g_{i,j,m}$. Then, note that $\bigcap_{f\in F}\LAB^f\ne \emptyset$ if and only if $\{t(f)\mid f\in F\}$ is a singleton, and in this case $\bigcap_{f\in F}\LAB^f=\LAB^{h_{t(f)}}$. On the other hand, $\bigcap_{g\in G}\LAB \setminus \LAB^g=\LAB\setminus \bigcup_{g\in G}\LAB^g$, so that $\LAB^{F,G}\ne \emptyset$ if and only if $\{t(f)\mid f\in F\}$ is a singleton and $t(f)\not\in \{ t(g)\mid g\in G\}$, and in this case $\LAB^{F,G}=\LAB^{h_{t(f)}}$. Hence, $q_{F,G}\ne 0$ if and only if $\{t(f)\mid f\in F\}$ is a singleton and $t(f)\not\in \{ t(g)\mid g\in G\}$, and in this case $q_{F,G}=q_{f}$.
}
\end{noname}

\begin{lemma}\label{Lem:TightRep=Katsura1}
Given any tight representation $\pi$ of $\SAB$ on a Hilbert space $\mathcal{H}$, we have that:
\begin{enumerate}
\item $q_{h_i}^{\pi}=p_{h_i}^{\pi}=q_{g_{j,i,n}}^{\pi}$ for all $1\leq i,j\leq N$, all $i\in \OmA(j)$ and all $n\in \Z$.
\item $q_{g_{t,i,m}}^{\pi}=\sum\limits_{j\in \OmA(i)} \sum\limits_{n=1}^{A_{i,j}} p_{g_{i,j,n}}^{\pi}$ for all $1\leq i,t\leq N$, all $i\in \OmA(t)$ and all $m\in \Z$.
\item $q_{h_i}^{\pi}\perp q_{h_j}^{\pi}$ if $i\ne j$.
\item $\pi_{h_i}\pi_{g_{i,j,n}}=\pi_{g_{i,j,n+B_{i,j}}}$ and $\pi_{g_{i,j,n}}\pi_{h_j}=\pi_{g_{i,j,n+A_{i,j}}}$ for all $1\leq i\leq N$,all $j\in \OmA(i)$ and all $n\in \Z$.
\end{enumerate}
\end{lemma}
\begin{proof}
Given a tight representation $\pi$ of $\SAB$ on a Hilbert space $\mathcal{H}$, let us compute three particular cases of the identity $\bigvee_{h\in H}p_h=q_{F,G}$ above:
\begin{enumerate}
\item If $F=\{ h_i\}$ and $G=\emptyset$, then the partition $H=\{ h_i\}$ give us $q_{h_i}^{\pi}=p_{h_i}^{\pi}$.
\item If $F=\{g_{j,i,n}\}$ for some $(j,i)\in \OmA$ and $1\leq n\leq A_{j,i}$, and $G=\emptyset$, then the partition $H=\{h_i\}$ give us $q_{g_{j,i,n}}^{\pi}=p_{h_i}^{\pi}$
\item If $F=\{ h_i\}$ and $G=\emptyset$, then the partition $H=\{ g_{i,j,n}\mid j\in \OmA(i), 1\leq n\leq A_{i,j}\}$ give us $q_{h_i}^{\pi}=\sum\limits_{j\in \OmA(i)} \sum\limits_{n=1}^{A_{i,j}} p_{g_{i,j,n}}^{\pi}$.
\end{enumerate}
Hence (1) and (2) hold, while (3) and (4) hold by definition of $\pi$ and of $\LAB$.
\end{proof}

In the reverse sense, we have the following result.

\begin{lemma}\label{Lem:TightRep=Katsura2}
The map $\tau: \LAB\rightarrow \OAB$ defined by:
\begin{enumerate}
\item[(i)] $\tau_{h_i}=u_i$ for all $1\leq i\leq N$.
\item[(ii)] $\tau_{g_{i,j,n}}=s_{i,j,n}$ for all $1\leq i\leq N$, all $j\in \OmA(i)$ and all $n\in \Z$.
\item[(iii)] If $g\in \LAB$ is written in standard form as $g=g_1\cdots g_n$, then $\tau_g=\tau_{g_1}\cdots \tau_{g_n}$.
\end{enumerate}
is a tight representation of $\LAB$ on a separable Hilbert space that respects least common multiples.
\end{lemma}
\begin{proof}
Consider the above defined map $\tau: \LAB\rightarrow \OAB$. Since the standard form is unique by Lemma \ref{Lem:UniqueNormalFormLAB}, and $s_{i,j,n}u_j=s_{i,j, n+A_{i,j}}$ and $u_is_{i,j,n}=s_{i,j, n+B_{i,j}}$ for all $(i,j)\in \OmA$ and $n\in \Z$, $\tau$ is a well-defined map. By definition, $\tau_f$ is a partial isometry for each $f\in \LAB$. The projections $p_f^{\tau}$ and $q_g^{\tau}$  are $s_fs_f^*$ and $q_{t(f)}$ respectively. Hence, they commute (see \ref{CuntzSemigroup}). Moreover, properties (i-iv) in Definition \ref{Def:SemigroupoidRep} are fulfilled, so that $\tau$ is a representation of $\LAB$ on $\OAB$.

In order to see that $\tau$ is a tight representation, recall that by \ref{SpecialCase}, $\LAB^{F,G}\ne \emptyset$ if and only if $\{t(f)\mid f\in F\}$ is a singleton and $t(f)\not\in \{ t(g)\mid g\in G\}$, and in this case $\LAB^{F,G}=\LAB^{h_{t(f)}}$. Hence, it is enough to show that, whenever $H$ is a finite partition of $\LAB^{F,G}$, then $q_{f}^{\tau}=\sum\limits_{h\in H}p_h^{\tau}$. But notice that by Definition \ref{Def:KatAlgAlgebra} we have $q_{h_i}^{\tau}=q_{g_{t,i,m}}^{\tau}=\sum\limits_{j\in \OmA(i)}\sum\limits_{n=1}^{A_{i,j}}p_{g_{i,j,,n}}^{\tau}$. Since any finite partition is obtained by expansion of the partitions $\{h_i\}$ and $\{g_{i,j,n}\mid j\in \OmA(i), 1\leq n\leq A_{i,j}\}$ (this is the idea of expansion of a rooted tree used in the proof of Proposition \ref{Prop:LAB-ExistsFinitePartition}), and these expansion faithfully transfer to the projections $p_f^{\tau}$ and $q_g^{\tau}$, we obtain the desired result. Also, because of Proposition \ref{Prop:LAB-ExistsFinitePartition} and \cite[Proposition 15.5]{Exel1}, $\tau$ respects least common multiples.
\end{proof}

Recall that a representation $\pi$ of a semigroupoid $\Lambda$ on a Hilbert space is said to be normal if it is tight and respects least common multiples. Then:
\begin{enumerate}
\item We will denote $\Oo_{\Lambda}$ the $C^*$-algebra generated by a universal tight representation $\pi$ of $\Lambda$ \cite[Definition 4.6]{Exel2}.
\item We will denote $\Oo_{\Lambda}^{\text{lcm}}$ the $C^*$-algebra generated by a universal normal representation $\pi^u$ of $\Lambda$ \cite[Definition 18.2]{Exel1}.
\end{enumerate}

As remarked in \cite[Section 18]{Exel1}, if $\Lambda$ satisfies the statement of Proposition \ref{Prop:LAB-ExistsFinitePartition}, then by \cite[Proposition 15.5]{Exel1} any tight representation is normal, and so $\Oo_{\Lambda}\cong \Oo_{\Lambda}^{\text{lcm}}$. Now, the above discussion allows us to prove the following result.

\begin{theorem}\label{Thm:OABisoOLAB}
The $C^*$-algebras $\OAB$ and $\Oo_{\LAB}$ are naturally isomorphic.
\end{theorem}
\begin{proof}
Because of Lemma \ref{Lem:TightRep=Katsura1} and the Universal Property of $\OAB$, there is a (unique) $\ast$-homomorphism $\varphi: \OAB \rightarrow \Oo_{\LAB}$ defined by $\varphi (u_i)=\pi_{h_i}$ and $\varphi (s_{i,j,n})=\pi_{g_{i,j,n}}$. On the other hand, by Lemma \ref{Lem:TightRep=Katsura2}  and the Universal Property of $\Oo_{\LAB}$, there exists a (unique) $\ast$-homomorphism $\psi: \Oo_{\LAB}\rightarrow \OAB$ defined by $\psi (\pi_{h_i})=u_i$ and $\psi (g_{i,j,n})=s_{i,j,n}$. As both maps are mutually inverse, we conclude that  $\OAB\cong \Oo_{\LAB}$.
\end{proof}

As a consequence we obtain the following result

\begin{corollary}\label{Cor:OABisoGroupoid}
If $\mathcal{G}_{\LAB}$ denotes the groupoid of germs associated to the tight part of the spectrum of $E(\SAB)$, then $\OAB\cong C^*(\mathcal{G}_{\LAB})$.
\end{corollary}
\begin{proof}
By Proposition \ref{Prop:LAB-Monic}, Proposition \ref{Prop:existsLCMinLAB} and Lemma \ref{Lem:LAB-CategoricalNoSprings} $\LAB$ fulfils the requirements of \cite[Theorem 18.4]{Exel1}. Hence, the result holds from \cite[Theorem 18.4]{Exel1} and Theorem \ref{Thm:OABisoOLAB}.
\end{proof}

\begin{remark}\label{Rem:Camins}
{\rm Notice that the universal tight representation $\pi:\LAB\rightarrow \Oo_{\LAB}$ extends to a representation $\pi: \SAB\rightarrow \OAB$ which is a monoid homomorphism (sending zero to zero). In fact, $\pi (\SAB)$ is exactly the inverse semigroup $\mathcal{S}^{A,B}$ of partial isometries of $\OAB$. Moreover, by Lemmas \ref{Lem:NormalFormLAB} and \ref{Lem:UniqueNormalFormLAB}, $\pi$ restricts to a monoid isomorphism between the submonoid
$$\Xi=\{ (f, \Lambda^f, 1)\mid f\in \LAB\}\cup \{ 0\}$$
of $\SAB$ and the submonoid
$$\Gamma=\{s_{i_1,i_2,n_1}\cdots s_{i_{k-1},i_k,n_{k-1}}u_{i_k}^{t_k}\mid 1\leq n_j\leq A_{i_j, i_{j+1}}, t_k\in \Z \}\cup \{u_i^l\mid 1\leq i\leq N, l\in \N\}\cup\{0\}$$
of $\mathcal{S}^{A,B}$.
}
\end{remark}

Finally, we are ready to prove the following result, which will be relevant to apply our techniques for studying the algebra $\OAB$. The first part of it appears in \cite[Proposition 5.6]{Kat1}.  Since the appearance of Brown and Ozawa's book \cite{B-O}, we may give it a much shorter proof.

\begin{theorem}\label{Prop: amenable}
Let $N\in \N\cup \{\infty\}$, and let $A\in M_N(\Z^+)$ and $B\in M_N(\Z)$ be matrices satisfying Condition (0). Then:
\begin{enumerate}
\item $\OAB$ is nuclear.
\item The groupoid $\mathcal{G}_{\LAB}$ is amenable.
\item $C^*_r(\mathcal{G}_{\LAB})=C^*(\mathcal{G}_{\LAB})$.
\end{enumerate}
\end{theorem}
\begin{proof}
$\mbox{ }$
\begin{enumerate}
\item By \cite[Subsection 4.3]{Kat1}, $\OAB$ is the Cuntz-Pimsner algebra $\mathcal{O}_{X_{A,B}}$ of a certain full $C^*$-correspondence $X_{A,B}$ over the commutative $C^*$-algebra $\mathcal{A}_N\cong C_0(\{1, 2, \dots , N\}\times \mathbb{T})$. Since $\mathcal{A}_N$ is nuclear, \cite[Theorem 4.6.25]{B-O} implies that the associated Toeplitz-Cuntz-Pimsner algebra $\mathcal{T}_{X_{A,B}}$ is nuclear. Since $\mathcal{O}_{X_{A,B}}$ is a quotient of $\mathcal{T}_{X_{A,B}}$ and nuclearity passes to quotients \cite[Theorem 9.4.4]{B-O}, we conclude that $\OAB$ is nuclear.
\item By Corollary \ref{Cor:OABisoGroupoid} we have that $\OAB\cong C^*(\mathcal{G}_{\LAB})$. Then, by part (1), $C^*(\mathcal{G}_{\LAB})$ is nuclear. Thus, by \cite[Theorem 9.4.4]{B-O}, so is the reduced groupoid $C^*$-algebra $C_r^*(\mathcal{G}_{\LAB})$, which is its quotient. Hence, $\mathcal{G}_{\LAB}$ is amenable by \cite[Theorem 5.6.18]{B-O}.
\item Since $\mathcal{G}_{\LAB}$ is amenable by part (2), the result holds by \cite[Corollary 5.6.17]{B-O}.
\end{enumerate}
\end{proof}

\section{Computing the space $\mathcal{G}_{\LAB}^{(0)}$}

In this section we will explicitly compute a suitable picture for $\mathcal{G}_{\LAB}^{(0)}$ directly associated to $\OAB$, and we will use it to give an alternative description of $\OAB$ in terms of an action $\alpha$ of $\mathcal{S}^{A,B}$ on the new picture of $\mathcal{G}_{\LAB}^{(0)}$.\vspace{.2truecm}

To this end recall that, according to \cite{Exel1}, $\mathcal{G}_{\LAB}^{(0)}$ is homeomorphic to the space $\widehat{E(\SAB)}_{\text{tight}}$ of tight characters on the semilattice $E(\SAB)$ of projections of $\SAB$.

In order to simplify the notation, along this section we will denote $\mathcal{S}:=\SAB$ and $E:=E(\SAB)$. Following the notation in (\ref{SAB}) and \cite[Notations 19.3]{Exel1}, we denote:
\begin{enumerate}
\item $E_p:=\{ p_f\mid f\in \LAB\}$.
\item $E_q:=\{ q_{h_i}\mid 1\leq i\leq N\}$.
\end{enumerate}

\begin{definition}[{\cite[Definition 19.5]{Exel1} }]
{\rm Given a filter $\xi$ in $E$, we will say that $\xi$ is of:
\begin{enumerate}
\item $p$-type, if $\xi\subseteq E_p$.
\item $q$-type, if $\xi\subseteq E_q$.
\item $pq$-type, if $\xi\cap E_p$ and $\xi\cap E_q$ are nonempty.
\end{enumerate}}
\end{definition}

\begin{definition}[{\cite[Definition 19.7]{Exel1}}]
Given a filter $\xi$ in $E$, we will say that the stem of $\xi$ is $\omega_{\xi}=\{ f\in \LAB\mid p_f\in \xi\}$.
\end{definition}

\begin{definition}[{\cite[Definition 19.10]{Exel1}}]
{\rm A path in $\LAB$ is a subset $\omega$ of $\LAB$ such that:
\begin{enumerate}
\item[(i)] If $f\in \omega$ and $g\in \LAB$ such that $g\vert f$, then $g\in \omega$;
\item[(ii)] For every $f,g\in \omega$ one has $f\Cap g$, and moreover $\text{lcm}(f,g)\in \omega$.
\end{enumerate}}
\end{definition}

In particular \cite[Proposition 19.9]{Exel1}, any stem $\omega_{\xi}$ of a filter $\xi$ is a path, and the correspondence $\xi\mapsto \omega_{\xi}$ defines a bijection between filters of $E$ (not of $q$-type) and paths of $\LAB$ \cite[Proposition 19.11]{Exel1}.

\begin{noname}\label{tight-filter}
{\rm Recall that if $f\in \LAB$, $r(f)$ denotes the (unique) $\LAB^{h_i}\in \mathcal{Q}$ such that $f\in \LAB^{h_i}$. Thus, under the above correspondence, \cite[Proposition 19.12]{Exel1} implies that the tight-filters $\xi$ (i.e. those whose associate character $\phi_{\xi}\in \widehat{E}_{\text{tight}}$) are exactly the $pq$-type filters $\xi$ satisfying:
\begin{enumerate}
\item For all $f\in \omega_{\xi}$ and for all $H\subseteq \LAB^f$ finite cover, there exists $h\in H$ such that $fh\in \omega_{\xi}$.
\item For all finite cover $H$ of $r(\omega_{\xi})$ one has $h\in \omega_{\xi}$ for some $h\in H$.
\end{enumerate}}
\end{noname}

\begin{proposition}\label{Prop:TightFitersLAB}
Let $\xi$ be a tight-filter of $E$, and let $\omega_{\xi}$ be the associated stem. Then, there exists a unique infinite sequence $\{ i_k\}_{k\geq 1}\subset \{1, 2 \dots ,N\}^{\N}$ and unique elements $1\leq n_j\leq A_{i_j,i_{j+1}}$ such that the elements in the stem $\omega_{\xi}$ are exactly of two types:
\begin{enumerate}
\item $h_{i_1}^t$, for every $t\in \N$.
\item $g_{i_1,i_2,n_1}g_{i_2,i_3,n_2}\cdots g_{i_{k-1},i_{k},n_{k-1}}g_{i_k,i_{k+1},m_k}$, for every $k\in \N$ and every $m_k\in \Z$ such that $m_k\equiv n_k (\text{mod } A_{i_k, i_{k+1}})$.
\end{enumerate}
\end{proposition}
\begin{proof}
First, let $\LAB^{h_{i_1}}=r(\omega_{\xi})$. Then, by Proposition \ref{Prop:LAB-ExistsFinitePartition} and property (2) in (\ref{tight-filter}):
\begin{enumerate}
\item[(i)] Using the partition $H=\{h_i\}$, we have that $h_{i_1}\in \omega_{\xi}$, and moreover, since $\omega_{\xi}$ is a path, it is the only $h_j\in \omega_{\xi}$, as otherwise $p_{h_{i_1}}\perp p_{h_j}$ with both projections in $\omega_{\xi}$ will imply that $0\in \xi$, which is impossible.
\item[(ii)] Given any $t\in \Z$, and using the partition $H_t=\{ g_{{i_1},j,n+tA_{{i_1},j}}\mid j\in \OmA(i_1), 1\leq i\leq A_{{i_1},j}\}$, we have that for each $t\in \Z$ there exists at least a $j_t\in \OmA(i)$ and a $1\leq n_t\leq A_{{i_1},j_t}$ such that $g_{{i_1},j_t,n_t+tA_{{i_1},j_t}}\in \omega_{\xi}$. And since $\omega_{\xi}$ is a path, by the same argument as in (i) there exist exactly one $j_t\in \OmA(i)$ and one $1\leq n_t\leq A_{{i_1},j_t}$ with this property. By the same reason, for $t\ne r\in \Z$, since both $g_{{i_1},j_t,n_t+tA_{{i_1},j_t}}$ and $g_{{i_1},j_r,n_r+rA_{{i_1},j_r}}$ are in $\omega_{\xi}$, we conclude by the argument in the proof of Proposition \ref{Prop:existsLCMinLAB} that $j_t=j_r$ and $n_t=n_r$. Thus, we fix $i_2=j_0$ and $n_1=n_0$.
\end{enumerate}

On the other hand, using these computations, jointly with by Proposition \ref{Prop:LAB-ExistsFinitePartition} and property (1) in (\ref{tight-filter}):
\begin{enumerate}
\item[(i)] Using the partition $H=\{h_i\}$, for all  $t\in \N$ we have that $h_{i_1}^t\in \omega_{\xi}$ (which fits with point (i) in the previous computation).
\item[(ii)] Using the partition $H=\{h_i\}$, for all $t\in \Z$, if $g_{{i_1},j,n+tA_{{i_1},j}}\in \omega_{\xi}$ then $g_{{i_1},j,n+tA_{{i_1},j}}h_j\in \omega_{\xi}$ (which fits with point (ii) in the previous computation).
\item[(iii)] Given any $t\in \Z$, and using the partition $H_t=\{ g_{{i_1},j,n+tA_{{i_1},j}}\mid j\in \OmA(i_1), 1\leq i\leq A_{{i_1},j}\}$, there exists a unique $j\in \OmA({i_1})$ and a unique $1\leq m_t\leq A_{{i_1},j}$ such that $h_{i_1} g_{{i_1},j,m_t+tA_{{i_1},j}}\in \omega_{\xi}$. Since $h_{i_1}g_{{i_1},j,m_t+tA_{{i_1},j}}=g_{{i_1},j,m_t+tA_{{i_1},j}+B_{{i_1},j}}$, point (ii) in the previous computation shows that $m_t=n_t-B_{{i_1},j}$.
\item[(iv)] Fixed the unique  element $g_{i_1,i_2,n_1}\in \omega_{\xi}$ obtained in the above computation, and using the partition $H_t=\{ g_{{i_1},j,n+tA_{{i_1},j}}\mid j\in \OmA(i_1), 1\leq i\leq A_{{i_1},j}\}$, we conclude by the same arguments as above that there exist a unique $i_3\in \{1,2, \dots , N\}$ and a unique $1\leq n_2\leq A_{i_2,i_3}$ such that $g_{i_1,i_2,n_1}g_{i_2,i_3,n_2}\in \omega_{\xi}$. Moreover, again as above, $g_{i_1,i_2,n_1}g_{i_2,j,m}\in \omega_{\xi}$ if and only if $j=i_3$ and $m\equiv n_k (\text{mod } A_{i_2, i_{3}})$.
\end{enumerate}
Recurrence on the last computation give us the desired result.
\end{proof}

\begin{noname}\label{IsUnique}
{\rm In particular, to each tight-filter $\xi$ of $E$ we can associate an infinite path
$$\widehat{\omega_{\xi}}:=g_{i_1,i_2,n_1}g_{i_2,i_3,n_2}\cdots$$
in the alphabet $\{g_{i,j,n}\mid 1\leq i\leq N, j\in \OmA(i), 1\leq n\leq A_{i,j}\}$, such that, for every $k\in \N$, the element $\widehat{\omega_{\xi}}\vert_{k}=g_{i_1,i_2,n_1}g_{i_2,i_3,n_2}\cdots g_{i_k,i_{k+1},n_k}$ belongs to $\omega_{\xi}$; clearly, $\widehat{\omega_{\xi}}$ is unique. We will denote by $X_A$ the set of infinite paths in the above alphabet, which is a totally disconnected locally compact Hausdorff space when endowed with the natural product topology. Also notice that, through the representation $\pi:\mathcal{S}\rightarrow \OAB$, the infinite path $\widehat{\omega_{\xi}}$ goes to the infinite path $\pi_{\widehat{\omega_{\xi}}}=s_{i_1,i_2,n_1}s_{i_2,i_3,n_2}\cdots$ of $X_A$ (and this is clearly a bijection).}
\end{noname}
\vspace{.2truecm}

In the reverse direction we have

\begin{definition}\label{Def:StemInfinitePath}
{\rm Let $\gamma=g_{i_1,i_2,n_1}g_{i_2,i_3,n_2}\cdots$ be an infinite path in the alphabet $\{g_{i,j,n}\mid 1\leq i\leq N, j\in \OmA(i), 1\leq n\leq A_{i,j}\}$. Then define
$$\omega(\gamma):=\{ h_{i_1}^t\mid t\in \N\}\cup \{ \gamma\vert_{k-1}g_{i_k,i_{k+1},m_k}\mid k\in \N, m_k\equiv n_k (\text{mod } A_{i_k, i_{k+1}})\}.$$
}
\end{definition}

Thus, we have

\begin{proposition}\label{Prop:TightPath4InfintePath}
The set $\omega(\gamma)$ is a tight path such that $\widehat{\omega(\gamma)}=\gamma$.
\end{proposition}
\begin{proof}
By definition it is clear that $f,g\in \omega(\gamma)$ implies $f\Cap g$. Also, it is clear that $f\in \omega(\gamma)$ and $g\in \LAB$ with $g\vert f$ implies $g\in \omega(\gamma)$. Hence, $\omega(\gamma)$ is a path. Moreover, the filter $\xi$ associated to $\omega(\gamma)$ is of $pq$-type.

Now, let us prove that $\xi$ is tight. For this we need to check properties (1-2) in (\ref{tight-filter}). It is enough to check it for finite partitions containing either $h_j$ or $g_{j,k,m}$. And that holds by definition of $\omega(\gamma)$. The final statement is obvious, so we are done.
\end{proof}

\begin{remark}\label{Rem:Nou}
{\rm Notice that $\omega$ defines a bijection between infinite paths in the alphabet $\{g_{i,j,n}\mid 1\leq i\leq N, j\in \OmA(i), 1\leq n\leq A_{i,j}\}$ and tight paths on $\LAB$. Moreover, by Remark \ref{Rem:Camins} and (\ref{IsUnique}), the bijection becomes a bijection between infinite paths in the alphabet $\{s_{i,j,n}\mid (i,j)\in \OmA, 1\leq n\leq A_{i,j}\}$ and tight paths in the monoid $\Gamma$.
}
\end{remark}

As a consequence of Remark \ref{Rem:Nou}, Proposition \ref{Prop:TightPath4InfintePath} and (\ref{IsUnique}), we obtain the following consequence

\begin{corollary}\label{Corol:Bijection}
There exists a bijection $\Phi:\widehat{E}_{\text{tight}}\rightarrow X_A$, defined by the rule $\Phi(\phi_{\xi})=\widehat{\omega_{\xi}}$.
\end{corollary}

Now, we will prove that $X_A$ is homeomorphic to $\widehat{E}_{\text{tight}}$.

\begin{noname}\label{Topology}
{\rm Recall that, whenever $\xi\subset E$ is a filter, then the associated character $\phi_{\xi}$ is given by the rule $\phi_{\xi}(x)=1$ if $x\in \xi$, and $\phi_{\xi}(x)=0$ otherwise. Hence, the topology on $\widehat{E}_{\text{tight}}$ is endowed by the inclusion $\widehat{E}_{\text{tight}}\subset \{ 0, 1\}^{E}$. In particular, the basic open neighborhoods of $\phi_{\eta}$ in $\widehat{E}_{\text{tight}}$ are of the form
$$W=\{ \phi_{\xi}\in \widehat{E}_{\text{tight}}\mid p_{t_1}, \dots ,p_{t_k}\in {\xi} \mbox{ and } p_{s_1}, \dots , p_{s_l}\not\in {\xi}\}$$
for some $t_1, \dots ,t_k, s_1, \dots ,s_l\in \mathcal{S}^{A,B}$.

On the other hand, the topology of $X_A$ is defined by the following fundamental system of open neighborhoods: for each $n\in \N$ and each $\gamma\in X_A$, the basic open (in fact clopen) sets are of the form
$$W_n^{\gamma}=\{ \tau \in X_A\mid \gamma\vert_n=\tau\vert_n\}.$$
}
\end{noname}

So, we have

\begin{proposition}\label{Prop:TopologyXA}
The map $\Phi:\widehat{E}_{\text{tight}}\rightarrow X_A$ is an homeomorphism of topological spaces.
\end{proposition}
\begin{proof}
First notice that, given an element $x\in \mathcal{S}^{A,B}$, it is either a projection in $E$, $u_i^t$ for some $t\in \N$, or can be written as $x=s_I u_{s(I)}^t s_J^*$ for $I,J$ multinindices of minimum length and $t\in \Z$. If $n$ denote the length of $J$, then given any $\gamma\in X_A$ the product $x\cdot \gamma\vert_n$ can be either zero, or $s_I u_{s(I)}^t$ if $\gamma= s_J\gamma'$. Hence, we can define the operation $x\cdot \gamma$ by zero is so is $x\cdot \gamma\vert_n$, or by $s_I u_{s(I)}^t\gamma'$ otherwise. Obviously, when $x$ is a projection, $x\cdot \gamma\ne 0$ if and only if $x\cdot \gamma=\gamma$. If $\xi\subset E$, then $x\in \xi$ if and only if $x\cdot\widehat{\omega_{\xi}}\ne 0$ if and only if $\phi_{\xi}(x)=1$.

By using the identification given by the map $\omega$ defined before and Remark \ref{Rem:Nou}, we can endow a topology on $X_A$ through the inclusion $X_A\subset \{ 0, 1\}^{\Gamma}$. In particular, the basic open neighborhoods of $\eta$ in $X_A$ are of the form
$$W=\{ \gamma \in X_A\mid {t_1}, \dots ,{t_k}\in \omega(\gamma) \mbox{ and } {s_1}, \dots , {s_l}\not\in \omega(\gamma)\}$$
for some $t_1, \dots ,t_k, s_1, \dots ,s_l\in \Gamma$. By the previous paragraph, we can express these open sets as
$$W=\{ \gamma \in X_A\mid p_{t_i}\cdot \gamma=\gamma \mbox{ and } p_{s_j}\cdot \gamma =0\}$$
for $t_1, \dots ,t_k, s_1, \dots ,s_l\in \Gamma$. Also, for $\gamma \in X_A$ and $n\in\N$, we can express
$$W_n^{\gamma}=\{ \tau \in X_A\mid p_{(\gamma\vert_n)}\cdot \tau=\tau\}.$$
In particular, defining $t_1=\gamma\vert _n$ and taking any $s_1\in \Gamma$ such that $s_1^*\cdot \gamma\vert_n=0$, we have that $W_n^{\gamma}$ is an open neighborhood of $\gamma$ of the form $W=\{ \gamma \in X_A\mid {t_1}\in \omega(\gamma) \mbox{ and } {s_1}\not\in \omega(\gamma)\}$.

Now, fix $W=\{ \gamma \in X_A\mid {t_1}, \dots ,{t_k}\in \omega(\gamma) \mbox{ and } {s_1}, \dots , {s_l}\not\in \omega(\gamma)\}$ an open neighborhood of $\eta\in X_A$; in particular $p_{t_i}\cdot \eta=\eta$ and $p_{s_j}\cdot \eta=0$. Set $n=\max\{ \vert t_1\vert, \dots, \vert t_k \vert ,\vert s_1\vert, \dots,\vert s_l\vert\}$; we will assume that the length of $u_i$ is zero. Thus, if $\omega \in W_n^{\eta}$, since $\omega=\eta\vert_n\cdot \widehat{\omega}$, we have that $p_{t_i}\cdot \omega=\omega$ and $p_{s_j}\cdot \omega=0$. Hence, $W_n^{\eta}\subset W$, whence, both topologies coincide on $X_A$.

Thus, for any $W=\{ \phi_{\xi}\in \widehat{E}_{\text{tight}}\mid p_{t_1}, \dots ,p_{t_k}\in {\xi} \mbox{ and } p_{s_1}, \dots , p_{s_l}\not\in {\xi}\}$ we have that $\Phi (W)=\{ \gamma \in X_A\mid {t_1}, \dots ,{t_k}\in \omega(\gamma) \mbox{ and } {s_1}, \dots , {s_l}\not\in \omega(\gamma)\}$. Hence, $\Phi:\widehat{E}_{\text{tight}}\rightarrow X_A$ is a continuous open bijective map, and thus an homeomorphism.
\end{proof}

The final step is to describe the action $\alpha: \mathcal{S}^{A,B}\rightarrow \widehat{E}_{\text{tight}}$ when considered as an action $\alpha: \mathcal{S}^{A,B}\rightarrow X_A$ via the homeomorphism $\Phi$.

\begin{noname}[{\cite[Section 10]{Exel1}}]\label{LastNote}
{\rm Let $s\in \mathcal{S}^{A,B}$, and consider two open sets of $\widehat{E}_{\text{tight}}$:
$$D_{s^*s}=\{ \phi \in \widehat{E}_{\text{tight}}\mid \phi(s^*s)=1\}$$
and
$$D_{ss^*}=\{ \phi \in \widehat{E}_{\text{tight}}\mid \phi(ss^*)=1\}.$$
Then, the action $\alpha: \mathcal{S}^{A,B}\rightarrow \widehat{E}_{\text{tight}}$ is defined by (partial) homeomorphisms
$$\Theta_s: D_{s^*s}\longrightarrow D_{ss^*}$$
under the rule $\Theta_s(\phi)(x)=\phi(s^*xs)$ \cite[Proposition 10.3]{Exel1}.}
\end{noname}
\vspace{.2truecm}

Now, we are ready to prove the main result of this section

\begin{theorem}\label{Thm:RightPartialAction}
Under the identification provided by the homeomorphism $\Phi:\widehat{E}_{\text{tight}}\rightarrow X_A$ of Proposition
\ref{Prop:TopologyXA}, the usual  action $\alpha: \mathcal{S}^{A,B}\rightarrow \widehat{E}_{\text{tight}}$ becomes an action $\alpha: \mathcal{S}^{A,B}\rightarrow X_A$ given by multiplication of elements of $X_A$ on the left by elements of $\mathcal{S}^{A,B}$.
\end{theorem}
\begin{proof}
Fix an element $s\in \mathcal{S}^{A,B}$. If we apply the homeomorphism $\Phi$ to the open sets $D_{s^*s}$ and $D_{ss^*}$, then we obtain open sets
$$X_{s^*s}=\{ \gamma\in X_A\mid s^*s\gamma=\gamma\}$$
and
$$X_{ss^*}=\{ \gamma\in X_A\mid ss^*\gamma=\gamma\}=\{ \gamma\in X_A\mid \gamma=s\gamma'\}.$$
Moreover, the maps
$$\alpha_s: X_{s^*s}\longrightarrow X_{ss^*}$$
given by $\alpha_s(\gamma)=s\cdot \gamma$
are well-defined homeomorphisms, which correspond to $\Theta$ under the identification given by $\Phi$. Hence, the usual action $\alpha$ of  $\mathcal{S}^{A,B}$ on $\widehat{E}_{\text{tight}}$ becomes the natural action of  $\mathcal{S}^{A,B}$ on $X_A$ given by multiplication on the left, as desired.
\end{proof}

Because of Proposition \ref{Prop:TopologyXA} and Theorem \ref{Thm:RightPartialAction}, we can apply \cite[Section 4]{Exel1} to give a useful picture of $\mathcal{G}_{\LAB}$. To be concrete

\begin{noname}\label{GroupoidPicture}
{\rm Consider the set
$$\Omega=\{ (s,x)\in \mathcal{S}^{A,B}\times X_A\mid x\in X_{s^*s}\}.$$
Given $(s,x), (t,y)\in \Omega$, we say that $(s,x)\sim (t,y)$ if $x=y$, and there exists $e\in E(\mathcal{S}^{A,B})$ such that $x\in X_e$ and $se=te$. By replacing $e$ by $es^*st^*t$, we can assume that $e\leq s^*s, t^*t$. Moreover, since $x=y$, we conclude that either $s^*s\leq t^*t$ or $t^*t\leq s^*s$.

Hence, $\mathcal{G}_{\LAB}$ is isomorphic to the groupoid of germs $\mathcal{G}:=\Omega/\sim$, where $\mathcal{G}^{(0)}$ is identified with $X_A$ by the rule $[e,x]\leftrightarrow x$ for any $x\in X_A$ and any $e\in E(\mathcal{S}^{A,B})$ such that $x\in X_e$. Under this identification, the maps $d,r:\mathcal{G}\rightarrow \mathcal{G}^{(0)}$ are given by $d[s,x]=x$ and $r[s,x]=s\cdot x$. The composition rule is $[s,x]\cdot [t,y]=[st, y]$ whenever $x=t\cdot y$, while $[s,x]^{-1}=[s^*, s\cdot x]$. Recall that the sets $\Omega (s,U):=\{ (s, \omega) \mid s\in \mathcal{S}^{A,B}, \omega \in U \}$, where $U\subset X_{s^*s}$ is an open set of $X_A$, give us a basis of slices (open compact sets) of $\mathcal{G}$.
}
\end{noname}

An interesting consequence is the following

\begin{corollary}\label{Corol:Etale}
The groupoid $\mathcal{G}_{\LAB}$ is \'etale with second countable unit space.
\end{corollary}
\begin{proof}
By Proposition \ref{Prop:TopologyXA}, $\mathcal{G}_{\LAB}^{(0)}$ is homeomorphic to $X_A$, which is a (locally) compact Hausdorff space. Also, $X_A$ is second countable. Finally, the picture of $\mathcal{G}_{\LAB}$ given in (\ref{GroupoidPicture}) proves that the map $r:\mathcal{G}_{\LAB}\rightarrow \mathcal{G}_{\LAB}^{(0)}$ is a local homeomorphism.
\end{proof}

So, we have the following result

\begin{proposition}\label{Prop:OABisoPartial}
The $C^*$-algebra $C^*(\mathcal{G}_{\LAB})$ is isomorphic to the inverse semigroup crossed product $C_{0}(\mathcal{G}_{\LAB}^{(0)})\rtimes_{\alpha} \mathcal{S}^{A,B}$.
\end{proposition}
\begin{proof}
By Corollary \ref{Corol:Etale}, $\mathcal{G}_{\LAB}$ is an \'etale groupoid with second countable unit space. Clearly, $\mathcal{S}^{A,B}$ is countable. Hence, the result holds from \cite[Proposition 9.9]{Exel1}.
\end{proof}

Thus, we obtain the following picture of $\OAB$

\begin{corollary}\label{Corol:RightPictureOAB}
The $C^*$-algebra $\OAB$ is isomorphic to the inverse semigroup crossed product $C_0(X_A)\rtimes_{\alpha} \mathcal{S}^{A,B}$.
\end{corollary}
\begin{proof}
The result is a direct consequence of Corollary \ref{Cor:OABisoGroupoid}, Proposition \ref{Prop:OABisoPartial}, Proposition \ref{Prop:TopologyXA} and Theorem \ref{Thm:RightPartialAction}.
\end{proof}

Notice that, when $B=(0)$, Corollary \ref{Corol:RightPictureOAB} recover the picture of the Exel-Laca algebra $\mathcal{O}_A$ given in \cite{ExelLaca}.

\section{Invariant sets and minimality}

In this section we will study the invariant open subsets of $X_A$ by the action of $\mathcal{S}^{A,B}$, and we characterize the minimality of $\mathcal{G}_{\LAB}$.

\begin{definition}
{\rm
If $S$ is a semigroup, and $\tau$ is an action by (partial) homeomorphisms on a topological space $X$, then:
\begin{enumerate}
\item We say that $x,y\in X$ are of equivalent trajectory under $\tau$, denoted $x\sim_{\tau}y$ if there exist $s,t\in S$ such that $\tau_s(x)$ and $ \tau_t(y)$ are defined and coincide.
\item We say that a subset $W$ of $X$ is invariant if for every $y\in W$ and $x\in X$, $x\sim_{\tau}y$ implies that $x\in W$.
\item We say that $X$ is irreducible if  it has no proper open invariant subsets.
\end{enumerate}
}
\end{definition}

In our case, we have the following

\begin{lemma}\label{Lem:ActionInvariant}
Given a subset $W$ of $X_A$, the following are equivalent:
\begin{enumerate}
\item $W$ is invariant.
\item For every $t\in \mathcal{S}^{A,B}$ we have that $\alpha_t(W\cap X_{t^*t})\subseteq W$.
\end{enumerate}
\end{lemma}
\begin{proof}$\mbox{ }$\vspace{.1truecm}

$(1)\Rightarrow (2)$. Let $y\in W$, and let $t\in \mathcal{S}^{A,B}$ with $x=\alpha_t(y)=t\cdot y$. Now, set $s=tt^*$, and notice that $s\cdot x$ is defined, and moreover $s\cdot x=tt^*\cdot x=tt^*t\cdot y=t\cdot y$. Hence, $\alpha_s(x)=\alpha_t(y)$, and thus $x\in W$.

$(2)\Rightarrow (1)$. Let $y\in W$ and let $x\in X$ such that $x\sim_{\tau}y$. Hence, there exists $s,t\in \mathcal{S}^{A,B}$ such that $s\cdot x=t\cdot y$. Then, we have $x=s^*t\cdot y=\alpha_{s^*t}(y)$, and thus $x\in W$.
\end{proof}

\begin{remark}\label{Rem:InvGroupoid}
{\rm By Lemma \ref{Lem:ActionInvariant} and (\ref{GroupoidPicture}), minimality of $\mathcal{G}_{\LAB}$ (see e.g. \cite{A} for a definition) faithfully translates to irredutibility of $X_A$ under the action of $\mathcal{S}^{A,B}$.
}
\end{remark}

\begin{definition}\label{Def:Orbit}
{\rm Given a subset $C$ of $X_A$, we define the orbit of $C$ to be the set
$$\text{Orb}(C)=\bigcup\limits_{t\in \mathcal{S}^{A,B}}\alpha_t(C\cup X_{t^*t}).$$
Notice that $\text{Orb}(C)$ is the smallest invariant subset of $X_A$ containing $C$.

For each $\gamma\in X_A$ and each $n\in \N$ we define
$$\Delta_n^{\gamma}=\text{Orb}(W_n^{\gamma}).$$
Hence, $\{\Delta_n^{\gamma}\mid \gamma\in X_A,n\in \N\}$ is a fundamental system of invariant open neighborhoods of $X_A$. In particular, if $x\in \Sigma:= \{s_{i,j,n}\mid (i,j)\in \OmA, 1\leq n \leq A_{i,j}\}\cup \{u_i\mid 1\leq i\leq N\}$, since $\alpha_x(X_{x^*x})=X_{xx^*}$ and $\alpha_{x^*}(X_{xx^*})=X_{x^*x}$, then both open sets have the same orbit, that we will denote $\Delta_x$. Given $x=u_i$, note that $u_i^*u_i=u_iu_i^*=p_i$, so that $\alpha_{x}(X_{xx^*})=X_{xx^*}=X_{x^*x}$, and $p_i=s_{j,i,n}^*s_{j,i,n}$, so that the action of $u_i$ do not give extra information when studying invariance of open subsets of $X_A$. Thus, we can restrict our attention to the action of elements in  $\Phi =\{s_{i,j,n}\mid (i,j)\in \OmA, 1\leq n \leq A_{i,j}\}$
}
\end{definition}

Now, we have the following

\begin{lemma}\label{Lem:Small-Inv}
If $U\subset X_A$ is a nonempty invariant open set, then there exists $x\in \Phi$ such that $\Delta_x\subseteq U$.
\end{lemma}
\begin{proof}
By (\ref{Topology}) there exists $n\in \N$ and $\gamma \in X_A$ such that $W_n^{\gamma}\subseteq U$. Now, define $t:=\gamma\vert_n$ and $x:=t_n$. Without lose of generality we can assume that $n>1$, so that $t\in \mathcal{S}^A$, and $x\in \Phi$.

If $\eta\in W_n^{\gamma}$, then $\eta=t\cdot \widehat{\eta}$, so that $\widehat{\eta}=t^*\cdot\eta$. Since $t^*t=x^*x$, we conclude that $\alpha_{t^*}(\eta)=\widehat{\eta}\in X_{x^*x}$. Conversely, if $\eta \in X_{x^*x}$, then $\eta=x^*x\cdot \eta=t^*t\cdot \eta$, so that $\alpha_t(\eta)=t\cdot \eta\in W_n^{\gamma}$. Hence, $W_n^{\gamma}$ and $X_{x^*x}$ are in the same orbit, so that
$$\Delta_x=\Delta_n^{\gamma}\subseteq \text{Orb}(U)=U,$$
as desired.
\end{proof}

\begin{lemma}\label{Lem:Inv-Connect}
If $x,y\in \Phi$ and there exists $t\in \mathcal{S}^A$ starting at $x$ and ending at $y$, then $\Delta_y\subseteq \Delta_x$.
\end{lemma}
\begin{proof}
Let $t\in \mathcal{S}^A$ such that $x=t_1$ and $y=t_{\vert t\vert}$. Now, if $\eta\in X_{y^*y}$, then $\eta=y^*y\cdot \eta$. Since $y^*y=t^*t$, we have that $\eta\in X_{t^*t}$, and thus $\alpha_t(\eta)=t\cdot\eta\in X_{xx^*}$. Taking orbits, we conclude that $\Delta_y\subseteq \Delta_x$, as desired.
\end{proof}

Given $x\in \Phi$, we denote $T_x=\{y\in \Phi\mid t(x)=d(y)\}$. Since $A$ is row-finite, $\vert T_x\vert< \infty$ for every $x\in \Phi$. Then

\begin{corollary}\label{Corol:Inv-Transitive}
For any $x\in \Phi$, $\Delta _x=\bigcup\limits_{y\in T_x}\Delta_y$.
\end{corollary}
\begin{proof}
By Lemma \ref{Lem:Inv-Connect}, $\bigcup\limits_{y\in T_x}\Delta_y\subseteq \Delta_x$. Now, let $\eta \in X_{x^*x}$. Since $\eta=x^*x\cdot \eta$, we have that $\eta=y\cdot \widehat{\eta}$ for some $y\in T_x$. Thus, $\eta\in X_{yy^*}$. Taking orbits, we conclude that $\Delta _x\subseteq\bigcup\limits_{y\in T_x}\Delta_y$, as desired.
\end{proof}

As immediate consequence, we get the following result

\begin{corollary}\label{Corol:Act-Transitive}
For the action $\alpha$ of $\mathcal{S}^{A,B}$ on $X_A$, the following are equivalent:
\begin{enumerate}
\item The matrix $A$ is irreducible.
\item The space $X_A$ is irreducible.
\item The groupoid $\mathcal{G}_{\LAB}$ is minimal.
\end{enumerate}
\end{corollary}
\begin{proof}$\mbox{ }$\vspace{.1truecm}

$(1)\Rightarrow (2)$ It is a direct consequence of Lemma \ref{Lem:Small-Inv} and Corollary \ref{Corol:Inv-Transitive}.

$(2)\Rightarrow (1)$ If $A$ is not irreducible, there exist $x,y\in \Phi$ such that no element of $\mathcal{S}^A$ starts at $x$ and ends at $y$. Hence, $\Delta_x$ turns out to be a proper open invariant subset, so we are done.

$(2)\Leftrightarrow (3)$ It is a consequence of Remark \ref{Rem:InvGroupoid}.
\end{proof}

\section{Essentially principal groupoids and topologically free actions}

In this section we will connect the notion of being topological free for $X_A$ with that of being essentially principal for $\mathcal{G}_{\LAB}$. Then, we will analyse fixed points and topological freeness of $X_A$. Given a subset $C$ of a topological space $X$, we denote its interior by $\text{Int}(C)$. Along this section we will assume that matrix $B$ satisfies Condition (E), so that $\mathcal{S}^{A,B}$ is a $E^*$-unitary semigroup (i.e. a $0$-$E$-semigroup).

\begin{definition}\label{Def:TopFree}
{\rm
Let  $S$ be an inverse  semigroup, and let $\tau$ be an action of $S$ on a topological space $X$.
\begin{enumerate}
\item Given $s\in S$ and $x\in X_{s^*s}$, we say  $x$ is a {fixed point} for $s$ if $\tau_s(x)=x$.
\item If $x$ is a  fixed point for $s$, we say that $x$ is a {trivial fixed point} if there exists $e\in E(S)$, such
that $x\in X_e$, and $se=e$.
\item We say that the action is {free} if  every fixed point for every $s$ in $S$ is trivial.
\item We say that the action is {topologically free} if, for every $s\in S$, the interior of the set of fixed
points for $s$ consists of trivial fixed points.
\end{enumerate}
}
\end{definition}

\begin{remark}\label{Remark: Ruy1}
{\rm $\mbox{ }$
\begin{enumerate}
\item It is easy to see that if $e$ is an idempotent elements in $S$, then every $x$ in $X_e$ is a trivial fixed
point for $e$.
\item On the other hand, when $S$ is $E^*$-unitary, it is not hard to show that for $s\in S\setminus
E(S)$, every fixed point of $s$ is nontrivial.
\end{enumerate}
}
\end{remark}

Thus, for the special case of $E^*$-unitary inverse semigroups we have:

\begin{theorem}\label{Theorem: Ruy 2}
Let $\tau$ be an action of an $E^*$-unitary inverse semigroup $S$ on a topological space $X$. Then:\begin{enumerate}
\item $\tau$ is free if and only if, for every $s\in S\setminus E(S)$, the set of fixed points for $s$ is empty.
\item $\tau$ is topologically free if and only if, for every $s\in S\setminus E(S)$, the set of fixed points for $s$ has
empty interior.
\end{enumerate}
\end{theorem}

We will connect topological freeness with a suitable notion for groupoids, which will help us to characterize simplicity for our algebras.

\begin{definition}\label{Def:EssPrinc}
{\rm Let $\mathcal{G}$ be a locally compact, Hausdorff, \'etale groupoid. Then:
\begin{enumerate}
\item For any $x\in \mathcal{G}^{(0)}$, the isotropy group at $x$ is
$$\mathcal{G}(x)=\{ \gamma\in \mathcal{G}\mid d(\gamma)=t(\gamma)=x\}.$$
\item $\mathcal{G}$ is principal if for every $x\in \mathcal{G}^{(0)}$ we have $\mathcal{G}(x)=\{x\}$.
\item $\mathcal{G}$ is essentially principal if the interior of the isotropy group bundle
$$\mathcal{G}' = \{\gamma\in\mathcal{G}: d(\gamma) = t(\gamma)\}$$
is contained in  $\mathcal{G}^{(0)}$.
\end{enumerate}
}
\end{definition}

Notice that, if $\mathcal{G}$ is second countable, Hausdorff, \'etale, and the unit space has the Baire property, then $\mathcal{G}$ is essentially principal if and only if for every $x\in \mathcal{G}^{(0)}$ the set $\{ x\in \mathcal{G}^{(0)}\mid \mathcal{G}(x)=\{x\}\}$ is dense in $\mathcal{G}^{(0)}$ \cite[Proposition 3.1]{R2}.\vspace{.2truecm}

Now, we have the following result, connecting both notions.

\begin{theorem}\label{Theorem: Ruy3}
Let $S$ be an $E^*$-unitary inverse semigroup, let $\tau$ be an action of $S$ on a locally compact,
Hausdorff space $X$, and let $\mathcal{G}$ be the corresponding groupoid of germs.  Then
$\mathcal{G}$ is essentially principal if and only if $\tau$ is topologically free.
\end{theorem}
\begin{proof}
As seen above, recall that  $\tau$ is topologically free if and only if,
\begin{quotation}
($\star$) for every $s\in S\setminus E(S)$, the interior of the set of fixed points for $s$ is empty.
\end{quotation}

Suppose that $\mathcal{G}$ is essentially principal, and let $s\in S\setminus E(S)$.  By way of contradiction suppose  that $x$ is an interior fixed point for $s$, so there exists an open subset $U$ of $X_{s^*s}$, consisting of
fixed points for $s$, such that $x\in U$.  It follows that $\Omega(s,U)$ is contained in $\mathcal{G}'$, and hence also in $\mathcal{G}^{(0)}$, by hypothesis.  In particular $[s,x]\in\mathcal{G}^{(0)}$, from where we deduce that
$[s,x]=[e,x]$, for some idempotent $e$ in $S$.

The condition for  equality of germs gives  an idempotent $f$ in $S$, such that $x\in X_f$, and $sf = ef$.  Since
$x\in X_e\cap X_f = X_{ef}$, it is clear that $ef$ is nonzero.  We conclude that $s$ dominates the nonzero idempotent $ef$, so $s$ itself is idempotent, a contradiction.

Conversely, assume that $\tau$ is topologically free, and let $\gamma$ lie in the interior of $\mathcal{G}'$.  One may then choose $s\in S$, and an open set $U\subseteq X_{s^*s}$, such that
$$\gamma\in\Omega(s,U)\subseteq \mathcal{G}'.$$
So $U$ is formed by fixed points for $s$, and hence $U$ is contained in the interior of the set of fixed points for $s$.  By ($\star$)  we may rule out  the possibility that $s\in S\setminus E(S)$, in turn concluding that $s\in E(S)$, and hence that $\gamma\in \mathcal{G}^{(0)}$.  This shows that $\mathcal{G}$ is essentially principal.
\end{proof}

\begin{remark}\label{Remark: Ruy4}
{\rm
Recall that, if the matrix $B$ satisfies Condition (E), then $\mathcal{S}^{A,B}$ is $E^*$-unitary, and thus Theorem \ref{Theorem: Ruy3} applies. This is the reason why we will require this property as hypothesis in the sequel.
}
\end{remark}

Since the property of being essentially principal plays a major role in deciding the simplicity of $C_r^*(\mathcal{G}_{\LAB})$ (see e.g. \cite[Theorem 5.1]{SimpleGroupoid}), Theorem \ref{Theorem: Ruy3} means that analysing fixed points and topological freeness for $X_A$ is an unavoidable step towards the characterisation of simplicity for $\OAB$ in this context.\vspace{.2truecm}

We will start by studying the action $\alpha$ restricted to the submonoid $\mathcal{S}^A$. Then, we have the following result

\begin{lemma}\label{Lem:RestrictedActionFixedPoints}
Let $t\in \mathcal{S}^A\setminus E(\mathcal{S}^{A,B})$, and let $\omega \in X_{t^*t}$ such that $\alpha_t(\omega)=\omega$. Then:
\begin{enumerate}
\item There exists finite multiindices $I\ne K$ such that $\omega=s_Is_Ks_K\cdots$.
\item The set $S_t=\{\eta \in X_A\mid \alpha_t(\eta)=\eta\}=\{ \omega\}$.
\end{enumerate}
\end{lemma}
\begin{proof}
Such an element $t$ must be either $s_I, s_J^*$ or $s_Is_J^*$ for finite multiindices $I\ne J$. Let us distinguish cases:
\begin{enumerate}
\item[(i)] If $t=s_I$, then $\omega=s_I\cdot \omega$. Hence, by recurrence, $\omega=s_Is_Is_I\cdots$.
\item[(ii)] If $t=s_J^*$, then $\omega=s_J^*\omega$ forces $\omega=s_J\widehat{\omega}$. Hence, $\widehat{\omega}=s_J^*s_J\widehat{\omega}=s_J^*\omega=\omega=s_J\widehat{\omega}$. Thus, $\omega=s_Js_Js_J\cdots$.
\item[(iii)] If $t=s_Is_J^*$, then $s_Is_J^*\cdot \omega=\omega=s_I\cdot \widehat{\omega}$, so that $s_J^*\cdot \omega=\widehat{\omega}$. But $\omega=s_Is_J^*\cdot \omega=(s_Is_J^*)^2\cdot \omega$ implies that $s_J^*s_I\ne 0$. Thus, either $J=IK$ or $I=JK$ for $K\ne \emptyset$. In the first case, $\omega=s_Is_K^*s_I^*\cdot \omega$ so that $\omega=s_Is_K\cdot \widehat{\omega}$. Then,
$$s_I\cdot \widehat{\omega}=\omega =s_Is_K^*s_I^*\cdot \omega=s_Is_K^*\cdot\widehat{\omega},$$
so that $\widehat{\omega}=s_K\cdot \widehat{\omega}$. Thus, $\omega =s_Is_Ks_K\cdots$. In the second case, an analogue argument works, and we get $\omega =s_Js_Ks_K\cdots$.
\end{enumerate}
So we are done.
\end{proof}

Recall that, given a matrix $A\in M_N(\Zplus)$, we define the associated directed graph $E_A$ by taking as a vertices the set $\{1, 2, \dots , N\}$ and as edges the set $\{s_{i,j,n}\mid (i,j)\in \OmA, 1\leq n\leq A_{i,j}\}$. We say that $E_A$:
\begin{enumerate}
\item Satisfies Condition (L) if every cycle has an exit.
\item Satisfies Condition (K) if every vertex in a cycle is a basis for at least two different cycles.
\end{enumerate}

Then, we have the following consequence

\begin{proposition}[{c.f. \cite[Proposition 12.2]{ExelLaca}}]\label{Prop:RestrictedTopolFree}
The restriction of the action $\alpha$ to $\mathcal{S}^A\setminus E(\mathcal{S}^{A,B})$ is topologically free if and only if the graph $E_A$ satisfies Condition (L).
\end{proposition}
\begin{proof}
Notice that $E_A$ fails Condition (L) if and only if there exists a finite sequence $i_1, i_2,\dots , i_k$ such that $(i_j,i_{j+1}), (i_k,i_1)\in \OmA$ and $A_{i_k,i_1}=A_{i_j,i_{j+1}}=1$ for all $j$.

Let $t\in \mathcal{S}^A\setminus E(\mathcal{S}^{A,B})$ be such that $S_t\ne \emptyset$. Then, by Lemma \ref{Lem:RestrictedActionFixedPoints}, there exist finite multiindices $I\ne K$ such that $S_t=\{\omega\}$ for the infinite path $\omega=s_Is_Ks_K\cdots$. Hence, $\text{Int}(S_t)=\emptyset$ if and only if $\omega$ is not an isolated point.

Suppose that $E_A$ satisfies Condition (L). Then, either there exists $l\ne i_{\vert I \vert +1}$ such that either $(i_{\vert I \vert +2}, l)\in \OmA$ or $A_{i_{\vert I \vert +1}, i_{\vert I \vert +2}}> 1$, so that for a suitable $m$ we have $s_{i_{\vert I \vert +1},l, m}$ is different of the initial section of $s_K$, $s_{i_{\vert I \vert +1},i_{\vert I \vert +2},n_1}$. Thus, given any $n\in \N$, there exists $r\in \N$ such that $\vert I\vert r\cdot \vert K\vert>n$. Hence, we can construct a path $\omega \ne \gamma=s_Is_K\cdots s_Ks_{i_1,l, m}\widehat{\gamma}$ such that $\gamma\in W_n^{\omega}$, whence $S_t\ne W_n^{\omega}$. Then, $\omega$ is not isolated.

Conversely, if $E_A$ fails Condition (L), then there exists a terminal circuit $s_I$. Let $t=s_I$ and $\omega=s_Is_Is_I\cdots$. So, $\omega \in X_{t^*t}$ and $\alpha_t(\omega)=t\cdot\omega=\omega$, whence $S_t=\{\omega\}$. Now, consider the open neighborhood $V=\{\eta \in X_A\mid p_t\cdot \eta=\eta\}$. Certainly, $\omega\in V$. Now, if $\gamma\in V$, the idetity $\gamma=p_t\cdot \gamma$ implies that $\gamma=t\cdot\widehat{\gamma}$. Since $t$ is a terminal circuit, we have that $\gamma=tttt\cdots= s_Is_Is_I\cdots=\omega$. Thus, $V=S_t$, so that $\text{Int}(S_t)\ne\emptyset$.
\end{proof}

Moreover

\begin{proposition}[{c.f. \cite[Proposition 12.4]{ExelLaca}}]\label{Prop:Condition(K)}
Let $t=s_Is_J^*\in \mathcal{S}^A\setminus E(\mathcal{S}^{A,B})$ and let $\omega=s_Ks_Ls_L\cdots$ a fixed point for $\alpha_t$. Then, $\omega$ is isolated in $\text{Orb}(\omega)$ if and only if the cycle $s_L$ is transitory in $E_A$ (i.e. no exit has a return path).
\end{proposition}
\begin{proof}
Suppose that $s_L$ is transitory, and consider the open neighborhood
$$V=\{\eta\in X_A\mid p_{s_Ks_L}\cdot \eta=\eta\}.$$
Let $\eta \in V\cap \text{Orb}(\omega)$, so that there exists $t\in \mathcal{S}^A$ such that $\eta=t\cdot \omega=ts_Ks_Ls_L\cdots$. But since
$$\eta= t\cdot \omega=ts_Ks_Ls_L\cdots=p_{s_Ks_L}\cdot \eta= p_{s_Ks_L}\cdot ts_Ks_Ls_L\cdots$$
and $s_L$ is transitory, we conclude that $\eta=\omega$. Thus, $V\cap \text{Orb}(\omega)=\{\omega\}$, whence $\omega$ is isolated in $\text{Orb}(\omega)$.

Conversely, suppose that $\omega$ is isolated in $\text{Orb}(\omega)$. Then, there exists an open neighborhood $U$ of $\omega$ such that $\{\omega\}=U\cap \text{Orb}(\omega)$. Hence, there exists $n\in\N$ such that $V=\{\eta \in X_A\mid p_{s_Ks_L^n}\cdot \eta=\eta\}\subseteq U$, and thus $\{\omega\}=V\cap \text{Orb}(\omega)$. If $s_L$ is not transitory, there exists a multiindex $T\ne L$ such that $\omega\ne \gamma=s_Ks_L^ns_Ts_Ls_L\cdots\in X_A$. Obviously, $\gamma\in V$. Take $t=s_Ks_L^ns_Ts_K^*\in \mathcal{S}^A$, and notice that $\alpha_t(\omega)=t\cdot \omega=\gamma$. Hence, $\gamma\in \text{Orb}(\omega)$, whence $\gamma\in V\cap \text{Orb}(\omega)$, and thus $\omega$ is not isolated.
\end{proof}

Consequently, the same argument used to prove \cite[Proposition 12.3]{ExelLaca} applies, and we conclude

\begin{proposition}
The restriction of the  action $\alpha$ to $\mathcal{S}^A\setminus E(\mathcal{S}^{A,B})$ is topologically free on every closed invariant subset of $X_A$ if and only if $E_A$ satisfies Condition (K).
\end{proposition}

Now, we will study the fixed points for the  action $\alpha_t$, when $t\in \mathcal{S}^{A,B}\setminus \mathcal{S}^A$, i.e. $t$ either equal $u_i^m$ or $s_Iu_{t(I)}^m$ for a finite multiindex $I$ and $m\in \Z$.

\begin{remark}\label{Rem:Bij=0-irrellevant}
{\rm Given $1\leq i\leq N$, if $B_{i,j}=0$ for every $j\in \OmA(i)$ (equivalently, if the $i$-th row of $B$ is identically zero), then by Definition \ref{Def:KatAlgAlgebra}(i) we have $u_is_{i,j,n}=s_{i,j,n}$ for every $j\in \OmA(i)$ and every $1\leq n\leq A_{i,j}$. Hence, Definition \ref{Def:KatAlgAlgebra} (i) and (iii) imply that $q_i=u_iq_i=u_i$, whence $u_i, u_i^*\in E(\mathcal{S}^{A,B})$. In particular, if $B=(0)$, we conclude that we only need to care about the restriction of $\alpha$ to $\mathcal{S}^A\setminus E(\mathcal{S}^{A,B})$; this is coherent with the fact that $\Oo_{A, (0)}\cong \Oo_A$.
}
\end{remark}

Thus, we will assume that $B\ne (0)$ and, when studying the action of $u_i^l$ on $X_A$ for a concrete $1\leq i\leq N$, that there exists at least one $j\in \OmA(i)$ such that $B_{i,j}\ne 0$.

\begin{lemma}\label{Lem:FixedU-sub-I}
Given an element $\omega=s_{i_1,i_2,n_1} s_{i_2,i_3,n_2}\cdots s_{i_k,i_{k+1},n_k}\cdots$ of $X_A$, the following are equivalent:
\begin{enumerate}
\item $\omega$ is fixed under the action of $u_{i_1}^l$ ($l\in \Z$).
\item For every $j\geq 1$ the element $K_j:=l\cdot \prod\limits_{t=1}^{j}\displaystyle\frac{B_{i_t,i_{t+1}}}{A_{i_t,i_{t+1}}}$ belongs to $\Z$.
\end{enumerate}
\end{lemma}
\begin{proof}
By Definition \ref{Def:KatAlgAlgebra}(i), $\omega =u_{i_1}^l\cdot \omega$ if and only if there exists a sequence $(K_j)_{j\geq 0}\subseteq \Z$ such that:\begin{enumerate}
\item[(i)] $K_0=l$.
\item[(ii)] For every $j\geq 1$, $n_{j-1}+K_{j-1}B_{i_j,i_{j+1}}=n_{j-1}+K_{j}A_{i_j,i_{j+1}}$.
\end{enumerate}
Notice that (ii) is equivalent to ask $K_{j-1}B_{i_j,i_{j+1}}=K_{j}A_{i_j,i_{j+1}}$ for every $j\geq 1$.

Now, for $j=1$ we have $l\cdot B_{i_1,i_{2}}=K_{1}A_{i_1,i_{2}}$, so that $K_1=l\cdot \displaystyle\frac{B_{i_1,i_{2}}}{A_{i_1,i_{2}}}$. Now, suppose that for $1\leq r\leq j-1$ we have proved that $K_j:=l\cdot \prod\limits_{t=1}^{r}\displaystyle\frac{B_{i_t,i_{t+1}}}{A_{i_t,i_{t+1}}}$. Hence
$$K_{j}A_{i_j,i_{j+1}}=K_{j-1}B_{i_j,i_{j+1}}=l\cdot \prod\limits_{t=1}^{r}\displaystyle\frac{B_{i_t,i_{t+1}}}{A_{i_t,i_{t+1}}}\cdot B_{i_j,i_{j+1}},$$
so that $K_j=l\cdot \prod\limits_{t=1}^{j}\displaystyle\frac{B_{i_t,i_{t+1}}}{A_{i_t,i_{t+1}}}$. This completes the proof.
\end{proof}

Let us show some characteristic examples of fixed elements.

\begin{example}\label{Exam:FixedU-sub-I}
{\rm $\mbox{ }$
\begin{enumerate}
\item Suppose that there exists $k\in \N$ such that $B_{i_k, i_{k+1}}=0$. Thus, given
$$\omega=s_{i_1,i_2,n_1} s_{i_2,i_3,n_2}\cdots s_{i_k,i_{k+1},n_k} s_{i_{k+1},j,m}\cdots ,$$
if we take $l=\prod\limits_{t=1}^{k}A_{i_t,i_{t+1}}$ then we have that $\omega$ is fixed by $u_{i_1}^l$ for any choice of sequences starting on $(i_{k+1},j)\in \OmA$. In particular, for any $n\geq k+1$, the open set $W_n^{\omega}$ is contained in $\{\omega \in X_A\mid u_i^l\cdot \omega =\omega\}$, and thus the set of fixed points has nonempty interior.
\item If $A\in M_N(\{0, 1\})$, then for any $\omega\in X_{q_i}$ we have that $u_i\cdot \omega =\omega$, so that for any $l\in \Z$ the set $\{\omega \in X_A\mid u_i^l\cdot \omega =\omega\}$ has nonempty  interior.
\item If there exists $l\in \N$ such that for any $(i,j)\in \OmA$ we have $B_{i,j}=l\cdot A_{i,j}$, then for any $\omega\in X_{q_i}$ we have that $u_i^l\cdot \omega =\omega$, so that these sets have nonempty  interior.
\end{enumerate}
}
\end{example}

\begin{remark}\label{Rem:NonIsolated}
{\rm Notice that if $\omega=s_{i_1,i_2,n_1} s_{i_2,i_3,n_2}\cdots s_{i_k,i_{k+1},n_k}\cdots$ of $X_A$ is fixed under the action of $u_{i_1}^l$ ($l\in \Z$), then so are all the elements $\omega'=s_{i_1,i_2,m_1} s_{i_2,i_3,m_2}\cdots s_{i_k,i_{k+1},m_k}\cdots$ for any choice of elements $1\leq m_j\leq A_{i_j, i_{j+1}}$ ($j\geq 1$). Thus, the set $\{\omega \in X_A\mid u_i^l\cdot \omega =\omega\}$ is not a singleton in general. Also notice that, by Proposition \ref{Prop:TopFreeExtended}, if the action of $u_i^l$ for all $1\leq i\leq N, l\in \Z$ is topologically free, then the matrix $B$ satisfies Condition (E).
}
\end{remark}

Now, we will look at topological freeness for this particular kind of actions.

\begin{proposition}\label{Prop:TopFreeExtended}
Given $t=u_i^l$ for $1\leq i\leq N$ and $l\in \Z$, and $\omega=s_{i_1,i_2,n_1} s_{i_2,i_3,n_2}\cdots s_{i_k,i_{k+1},n_k}\cdots$ a fixed element by $t$. Then, the following are equivalent:
\begin{enumerate}
\item $\text{Int}(\{\delta \in X_A\mid t\cdot \delta =\delta\})=\emptyset$.
\item For every $n\geq 1$ there exist $m\geq n$ and $j_{m+1}$ with:
\begin{enumerate}
\item $(i_m, j_{m+1})\in \OmA$ .
\item $K_{m+1}=K_m\cdot \displaystyle\frac{B_{i_m, j_{m+1}}}{A_{i_m, j_{m+1}}}\not\in \Z$.
\end{enumerate}
\end{enumerate}
\end{proposition}
\begin{proof}
We need to characterize when, for any $n\geq 1$, $W_n^{\omega}$ is not contained in $\{\delta \in X_A\mid t\cdot \delta =\delta\}$.

Since $\omega$ is fixed by $t$, by Lemma \ref{Lem:FixedU-sub-I} we have
$$K_j=l\cdot \prod\limits_{t=1}^{j}\displaystyle\frac{B_{i_t,i_{t+1}}}{A_{i_t,i_{t+1}}}\in \Z$$
for all $j\geq 0$. Thus, in order to get the condition above, for any $n\in \N$ we need to have $m\geq n$ and an alternative path $$\gamma=s_{i_1,i_2,n_1}\cdots s_{i_{m-1},i_{m},n_{m-1}}s_{i_{m},j_{m+1},p_{m}}\cdots  $$
such that $K_{m+1}=K_m\cdot \displaystyle\frac{B_{i_m, j_{m+1}}}{A_{i_m, j_{m+1}}}\not\in \Z$. Hence, the equivalence is clear.
\end{proof}

A practical situation under which Proposition \ref{Prop:TopFreeExtended} works is the following

\begin{corollary}\label{Corol:MesQueKatsura}
Given $t=u_i^l$ for $1\leq i\leq N$ and $l\in \Z$, we have that $\text{Int}(\{\omega \in X_A\mid t\cdot \omega =\omega\})=\emptyset$ (where $\omega$ denotes $s_{i_1,i_2,n_1} s_{i_2,i_3,n_2}\cdots s_{i_k,i_{k+1},n_k}\cdots$) if:
\begin{enumerate}
\item $B_{i_j,i_{j+1}}\ne 0$ for all $j$.
\item For every $n, r\geq 1$ there exist a sequence $j_{n+1}, j_{n+2}, \dots ,j_{n+r}$ with:
\begin{enumerate}
\item $(j_{n+t}, j_{n+t+1})\in \OmA$ for all $t$.
\item $\lim\limits_{r\rightarrow \infty}\prod\limits_{t=1}^{r}\left(\displaystyle\frac{B_{j_{n+t}, j_{n+t+1}}}{A_{j_{n+t}, j_{n+t+1}}}\right)=0$.
\end{enumerate}
\end{enumerate}
\end{corollary}
\begin{proof}
Notice that, given $n$, we have that $K_n$ is a fixed integer. So, by taking a sufficiently large sequence we can guarantee that $\displaystyle\frac{1}{K_n} >\prod\limits_{t=1}^{r}\displaystyle\frac{B_{j_{n+t}, j_{n+t+1}}}{A_{j_{n+t}, j_{n+t+1}}}$, so that $K_{n+r}\not\in \Z$.
\end{proof}

An example of application of Corollary \ref{Corol:MesQueKatsura} is the following

\begin{example}\label{Exam:MesQueKatsura}
{\rm Suppose that $B$ satisfies Condition (E) and that for every $1\leq i\leq N$ we have that $\vert B_{i,i}\vert <\vert A_{i,i}\vert$. Then Corollary \ref{Corol:MesQueKatsura} applies for every $t=u_i^l$ for $1\leq i\leq N$ and $l\in \Z$, because we can fix $j_{n+t}=i_{n}$ for all $t$, and then
$$\lim\limits_{r\rightarrow \infty}\prod\limits_{t=1}^{r}\left(\displaystyle\frac{B_{j_{n+t}, j_{n+t+1}}}{A_{j_{n+t}, j_{n+t+1}}}\right)=\lim\limits_{r\rightarrow \infty} \left(\displaystyle\frac{B_{i_{n}, i_{n}}}{A_{i_{n}, i_{n}}}\right)^r=0.$$
So, the action of $\alpha$ restricted to elements of the form $t=u_i^l$ for $1\leq i\leq N$ and $l\in \Z$ is topologically free. In particular, under the additional hypothesis of $B$ satisfying Condition (E), Corollary \ref{Corol:MesQueKatsura} applies for Katsura's conditions (\ref{KatsuraConditions}(2)) \cite[Proposition 2.10]{Kat1}.
}
\end{example}

\begin{lemma}\label{Lem:FixedMixedU-sub-I}
Given an element $\omega=s_{i_1,i_2,n_1} s_{i_2,i_3,n_2}\cdots s_{i_k,i_{k+1},n_k}\cdots$ of $X_A$, the following are equivalent:
\begin{enumerate}
\item $\omega$ is fixed under the action of $s_Iu_{t(I)}^l$ ($l\in \Z$ and $I$ multiindex).
\item We have that:
\begin{enumerate}
\item[(i)] $s_I$ is a cycle of the form $s_{i_1,i_2,n_1} s_{i_2,i_3,n_2}\cdots s_{i_k,i_{k+1},n_k}$ for some $k\in \N$.
\item[(ii)] There exists a sequence $(s_I^{(j)})_{j\geq 1}$ of cycles with $s_I^{(0)}=s_I$, and a sequence $(T_j)_{j\geq 0}$ of integers with $T_0=l$, such that $u_{t(I)}^{T_j}s_I^{(j)}=s_I^{(j+1)}u_{t(I)}^{T_{j+1}}$ for all $j\geq 1$, and $\omega= s_I s_{I}^{(1)}\cdots s_{I}^{(k)}\cdots$.
\end{enumerate}
\end{enumerate}
\end{lemma}
\begin{proof}
Sufficiency is clear. Let us proof necessity. Since $\omega= s_Iu_{t(I)}^l\cdot \omega$, we have that $\omega =s_I\cdot \omega{(1)}$, whence (i) holds. Now,
$$s_I\cdot \omega^{(1)}=\omega =s_Iu_{t(I)}^l\cdot \omega=s_Iu_{t(I)}^ls_I\cdot \omega^{(1)},$$
so that $\omega^{(1)}=u_{t(I)}^ls_I\cdot \omega^{(1)}=s_I^{(1)}u_{t(I)}^{T_1} \omega^{(1)}=s_I^{(1)}\cdot \omega^{(2)}$. Recurrence on this argument proves that (ii) holds, so we are done.
\end{proof}

\begin{remark}\label{esunic}
{\rm 
Notice that, as it occurs in Lemma \ref{Lem:RestrictedActionFixedPoints}, the element $\omega$ fixed by $t=s_Iu_{t(I)}^l$ is uniquely determined by $t$, and thus $\{\delta\in X_A\mid t\cdot \delta =\delta\}=\{\omega \}$.
}
\end{remark}

\begin{proposition}\label{Prop:TopFreeExtended2}
Given $t=s_Iu_i^l$ for $I\in E_A^*$, $1\leq i\leq N$ and $l\in \Z$, and $\omega=s_{i_1,i_2,n_1} s_{i_2,i_3,n_2}\cdots s_{i_k,i_{k+1},n_k}\cdots$ a fixed element by $t$. Then, the following are equivalent:
\begin{enumerate}
\item $\text{Int}(\{\delta \in X_A\mid t\cdot \delta =\delta\})=\emptyset$.
\item The graph $E_A$ satisfies Condition (L).
\end{enumerate}
\end{proposition}
\begin{proof}
Essentially, the proof is analog to that of Proposition \ref{Prop:RestrictedTopolFree}.

$(1)\Rightarrow (2)$ Suppose that $E_A$ do not satisfy Condition (L). Then, there exists a cycle without exits $s_I$. In particular, for $i=s(I)$ and any $l\in \Z$ we have that $u_i^l s_I=s_I u_i^{T_l}$ for some $T_l\in \Z$. Hence, the element $\omega =s_Is_I\cdots$ is fixed by $t=s_Iu_i^l$, and moreover for any $n\in \N$ we have that $W_n^{\omega}=\{\omega\}$. Thus, $(1)$ fails.

$(2)\Rightarrow (1)$ Suppose that $\omega =s_Is_I^{(1)}\cdots s_I^{(k)}\cdots$ is fixed by $t=s_I u_i^l$ for some $l\in \Z$. If $E_A$ satisfies Condition (L) and $s_I=s_{i_1,i_2,n_1} s_{i_2,i_3,n_2}\cdots s_{i_k,i_1,n_k}$, there exists $1\leq m\leq k$ and $j\in \Omega(i_m)\setminus \{i_{m+1}\}$. Hence, given any $n\in \N$ there is $r$ large enough such that $r\cdot \vert I\vert >n$, and the cycle $s_I^{(k+1)}$ has an exit $s_{i_m, j, p}$ with $1\leq p\leq A_{i_m, j}$. For any infinite path $\gamma\in X_{p_j}$ we have that
$$\tau:= s_Is_I^{(1)}\cdots s_I^{(r)}s_{i_1,i_2,n_1} s_{i_2,i_3,n_2}\cdots s_{i_{m-1},i_m,n_{m-1}}s_{i_m, j, p}\gamma \in W_n^{\omega},$$
and by Remark \ref{esunic} $\tau$ is not fixed by $t$. Thus $(1)$ holds. 
\end{proof}

\begin{remark}\label{Rem:general element}
{\rm For the case of elements $s=s_Iu_{r(I)}^ts_J^*$, if $\omega\in X_A$ and $s\cdot \omega=\omega$, then it is clear that $\omega=s_J\cdot \eta$ for some $\eta\in X_A$, and also that $s_J^*s_I\ne 0$, whence either $s_I=s_Js_K$ or $s_J=s_Is_K$ for a suitable multiindex $K$. In the first case, $s\cdot \omega=\omega $ if and only if $\eta= s_Ku_{r(I)}^t\cdot \eta$. In the second case, $s\cdot \omega=\omega $ if and only if $\eta= s_{\widehat{K}}u_{r(\widehat{K})}^l\cdot \eta$, where $s_{\widehat{K}}u_{r(\widehat{K})}^l=u_{r(I)}^{-t}s_K$. So, in any case, the problem reduces to applying Lemma \ref{Lem:FixedMixedU-sub-I}.
}
\end{remark}

Hence, we conclude that

\begin{theorem}\label{Thm:EssentiallyPrincipal}
Let $\alpha$ be the  action of $\mathcal{S}^{A,B}$ on $X_A$, and let $\mathcal{G}_{\LAB}$ the associated groupoid. The following are equivalent:
\begin{enumerate}
\item
\begin{enumerate}
\item The graph $E_A$ satisfies Condition (L).
\item The matrix $B$ satisfies Condition (E).
\item For any fixed point  $\omega=s_{i_1,i_2,n_1} s_{i_2,i_3,n_2}\cdots s_{i_k,i_{k+1},n_k}\cdots$ and every $n\geq 1$ there exist $m\geq n$ and $j_{m+1}$ with:
\begin{enumerate}
\item $(i_m, j_{m+1})\in \OmA$ .
\item $K_{m+1}=K_m\cdot \displaystyle\frac{B_{i_m, j_{m+1}}}{A_{i_m, j_{m+1}}}\not\in \Z$.
\end{enumerate}
\end{enumerate}
\item The groupoid $\mathcal{G}_{\LAB}$ is essentially principal.
\end{enumerate}
\end{theorem}
\begin{proof}
By Proposition \ref{Prop:RestrictedTopolFree}, Proposition \ref{Prop:TopFreeExtended}, Proposition \ref{Prop:TopFreeExtended2} and Remark \ref{Rem:general element}, $(1)$ is equivalent to the action of $\mathcal{S}^{A,B}$ being topologically free. By Remark \ref{Remark: Ruy4} and Theorem \ref{Theorem: Ruy3}, this is equivalent to the groupoid $\mathcal{G}_{\LAB}$ being essentially principal, as desired. 
\end{proof}

And as a practical consequence:

\begin{proposition}\label{Prop:EssentiallyPrincipal}
Let $\alpha$ be the action of $\mathcal{S}^{A,B}$ on $X_A$, and let $\mathcal{G}_{\LAB}$ the associated groupoid. If
\begin{enumerate}
\item The graph $E_A$ satisfies Condition (L).
\item The matrix $B$ satisfies Condition (E).
\item For any fixed point  $\omega=s_{i_1,i_2,n_1} s_{i_2,i_3,n_2}\cdots s_{i_k,i_{k+1},n_k}\cdots$ and for every $n, r\geq 1$ there exist a sequence $j_{n+1}, j_{n+2}, \dots ,j_{n+r}$ with:
\begin{enumerate}
\item $(j_t, j_{t+1})\in \OmA$ for all $t$.
\item $\lim\limits_{r\rightarrow \infty}\prod\limits_{t=1}^{r}\left(\displaystyle\frac{B_{j_{n+t}, j_{n+t+1}}}{A_{j_{n+t}, j_{n+t+1}}}\right)=0$.
\end{enumerate}
\end{enumerate}
then the groupoid $\mathcal{G}_{\LAB}$ is essentially principal.
\end{proposition}

\section{Simplicity of $\OAB$}

In this section we use the results in the previous sections to characterize the simplicity of $\OAB$. Notice that, by Theorem \ref{Prop: amenable}, it is enough to show simplicity for $C_{r}^*(\mathcal{G}_{\LAB})$. The central result is

\begin{theorem}\label{Thm:Simple}
Consider the initial matrices $A,B$. If the matrix $B$ satisfies Condition (E), then the following are equivalent:
\begin{enumerate}
\item
\begin{enumerate}
\item The matrix $A$ is irreducible.
\item The graph $E_A$ satisfies Condition (L).
\item For any fixed point  $\omega=s_{i_1,i_2,n_1} s_{i_2,i_3,n_2}\cdots s_{i_k,i_{k+1},n_k}\cdots$ and every $n\geq 1$ there exist $m\geq n$ and $j_{m+1}$ with:
\begin{enumerate}
\item $(i_m, j_{m+1})\in \OmA$ .
\item $K_{m+1}=K_m\cdot \displaystyle\frac{B_{i_m, j_{m+1}}}{A_{i_m, j_{m+1}}}\not\in \Z$.
\end{enumerate}
\end{enumerate}
\item $\OAB$ is simple.
\end{enumerate}
\end{theorem}
\begin{proof}
By Remark \ref{Rem:Cond(E)} and Corollary \ref{Corol:Etale} the groupoid $\mathcal{G}_{\LAB}$ is Hausdorff, \'etale with second countable unit space. By Corollary \ref{Corol:Act-Transitive} and Theorem \ref{Thm:EssentiallyPrincipal}, conditions in point (1) are equivalent to $\mathcal{G}_{\LAB}$ being a minimal  essentially principal groupoid. These three facts jointly are equivalent to $C_{r}^*(\mathcal{G}_{\LAB})$ being simple \cite[Theorem 5.1]{SimpleGroupoid}. Thus, Theorem \ref{Prop: amenable} give us the desired result.
\end{proof}

And as a direct consequence we have:

\begin{proposition}\label{Prop:Simple}
Consider the initial matrices $A,B$. If
\begin{enumerate}
\item The matrix $A$ is irreducible.
\item The graph $E_A$ satisfies Condition (L).
\item The matrix $B$ satisfies Condition (E).
\item For any fixed point  $\omega=s_{i_1,i_2,n_1} s_{i_2,i_3,n_2}\cdots s_{i_k,i_{k+1},n_k}\cdots$ and for every $n, r\geq 1$ there exist a sequence $j_{n+1}, j_{n+2}, \dots ,j_{n+r}$ with:
\begin{enumerate}
\item $(j_t, j_{t+1})\in \OmA$ for all $t$.
\item $\lim\limits_{r\rightarrow \infty}\prod\limits_{t=1}^{r}\left(\displaystyle\frac{B_{j_{n+t}, j_{n+t+1}}}{A_{j_{n+t}, j_{n+t+1}}}\right)=0$.
\end{enumerate}
\end{enumerate}
then $\OAB$ is simple.
\end{proposition}

\begin{proof}
Replace Theorem \ref{Thm:EssentiallyPrincipal} by Proposition \ref{Prop:EssentiallyPrincipal} in the proof of Theorem \ref{Thm:Simple}.
\end{proof}

As a corollary we have

\begin{corollary}\label{Corol:Simple}
Consider the initial matrices $A,B$. If they satisfy (\ref{KatsuraConditions}(2)) and $B$ satisfies Condition (E), then $\OAB$ is simple.
\end{corollary}

\begin{remark}\label{Rem:SimpleLlevatAmenable}
{\rm  Notice that we need an extra property --Condition (E)-- to obtain a characterization result for simplicity of $\OAB$, so that our results do not give a complete characterization of simplicity for $\OAB$. Nevertheless, our strategy will allow us to describe simplicity of $\OAB$ for a broad collection of algebras, including the ones given by Katsura. Moreover, the results are obtained in a more natural way, by linking this property to dynamical properties of $X_A$.
}
\end{remark}

\section{Pure infiniteness of $\OAB$}

In this section we use the results in the previous sections to give sufficient conditions for the algebra $\OAB$ being purely infinite (simple). As in Section 8 notice that, by Theorem \ref{Prop: amenable}, it is enough to show simplicity for $C_{r}^*(\mathcal{G}_{\LAB})$. For this we need to recall a slight reformulation of a definition from \cite{A}.

\begin{definition}[{\cite[Definition 2.1]{A}}]\label{Def: LocallyContracting}
{\rm We say that a second countable \'etale groupoid $\mathcal{G}$ is locally contracting in for every nonempty open subset $U$ of $\mathcal{G}^{(0)}$ there exists an open subset $V$ in $U$ and an slice $S$ such that $\overline{V}\subset S^{-1}S$ and $S\overline{V}S^{-1}$ is properly contained in $V$.
}
\end{definition}

In our case, using \cite[Proposition 4.18]{Exel1} and (\ref{GroupoidPicture}) we have that, for any $s\in \mathcal{S}^{A,B}$ and for any $U\subseteq X_{s^*s}$, the sets
$\Omega (s, U)=\{[s,x]\in \mathcal{G}_{\Lambda_{A,B}} \mid x\in U\}$
are slices (in fact they form a basis for the topology of $\mathcal{G}_{\Lambda_{A,B}}$). In view of that, we can restrict our attention to this kind of slices. Given any such slice $T:=\Omega (s, U)$, it is easy to see that for any $x\in \overline{V}\subset T^{-1}T$ the unique element of $T$ which can act by conjugation on $x$ is $[s,x]$, and so $[s, x]\cdot x\cdot [s,x]^{-1}=s\cdot x$. Under this reduction, we can understand $T\overline{V}T^{-1}$ as the set $\alpha_s(\overline{V})$ of images of $\overline{V}$ by the action $\alpha_s$ performed by the element $s\in\mathcal{S}^{A,B}$ defining the slice. Also, recall that $\{W_n^{\gamma}\mid n\in \N , \gamma\in X_A\}$ is a basis of clopen sets of $X_A$.

\begin{noname}\label{CondLocContr}
{\rm Thus, in order to check local contractiveness of $\mathcal{G}_{\Lambda_{A,B}}$ it is enough to show that for any $n\in \N$ and any $\gamma\in X_A$ there exist $n<m\in \N$, $\delta\in X_A$ and $s\in \mathcal{S}^{A,B}$ such that:
\begin{enumerate}
\item $W_m^{\delta}\subseteq W_n^{\gamma}$ (equivalently, $\delta\vert _n=\gamma\vert _n$).
\item $W_m^{\delta}\subseteq X_{s^*s}$.
\item For any $\omega \in W_m^{\delta}$ we have $\alpha_s(\omega)=s\cdot \omega \in W_m^{\delta}$.
\item There exists $\widehat{\omega}\in W_m^{\delta}$ such that $\widehat{\omega}\ne \alpha_s(\omega)=s\cdot \omega $ for every $\omega \in W_m^{\delta}$.
\end{enumerate}}
\end{noname}

A subsequent reduction can be done on the choice of $s\in \mathcal{S}^{A,B}$. To be concrete, we only need to analyze the following cases:
\begin{enumerate}
\item $s=u_i$, for any $1\leq i\leq N$.
\item $s=u_i^*$, for any $1\leq i\leq N$.
\item $s=s_I$ for any multiindex $I$.
\item $s=s_I^*$ for any multiindex $I$.
\end{enumerate}
because any $s\in \mathcal{S}^{A,B}$ can be written using the above elements. Let us study case by case.

\begin{lemma}\label{Lem:case u_i}
Let $n\in \N$, $\gamma\in X_A$. Then, $\alpha_{u_i}(W_n^{\gamma})\subseteq W_n^{\gamma}$ implies that $s(\gamma\vert_n)=i, r(\gamma\vert_n)=j$ and there exists $l\in \Z$ such that $u_i(\gamma\vert_n)=(\gamma\vert_n)u_j^l$. In particular, the above inclusion is always an equality.
\end{lemma}
\begin{proof}
Since $u_i=u_ip_i$, it is clear that $s(\gamma\vert_n)=i$. Now, set $r(\gamma\vert_n)=j$. By definition, for any $\omega \in W_n^{\gamma}$ we have $\omega\vert_n =\gamma\vert_n$. Hence, $u_i\cdot\omega \in W_n^{\gamma}$ means that for any $\delta \in X_{p_j}$ there exists $l\in \Z$ such that $u_i(\gamma\vert_n)\cdot \delta = ((\gamma\vert_n)u_j^l)\cdot \delta$.

In order to see that the above inequality is an equality it is enough to observe that, for any $\delta \in X_{p_j}$, the element $\eta :={(u_i^*)}^l\cdot \delta \in X_{p_j}$ and $u_i(\gamma\vert_n)\cdot \eta = ((\gamma\vert_n)u_j^l)\cdot \delta$.
\end{proof}

It is clear that the same argument applies to $s=u_i^*$. So, in order to check local contractiveness of $\mathcal{G}_{\Lambda_{A,B}}$ we only need to pay attention to the remaining cases.

\begin{lemma}\label{Lem:case s_I}
Let $n\in \N$, $\gamma\in X_A$. Then, for any $n=k\cdot \vert I \vert$, $\alpha_{s_I}(W_n^{\gamma})\subseteq W_n^{\gamma}$ implies that $s(\gamma\vert_n)=r(I)$, $s_I$ is a cycle and $\gamma\vert_n=s_Is_I\cdots s_I$ ($k$ times). Moreover, the above inclusion is strict if and only if the cycle $s_I$ has an exit.
\end{lemma}
\begin{proof}
It is obvious that $s(\gamma\vert_n)=r(I)$. On the other side, if for any $\omega \in W_n^{\gamma}$ we have $s_I\cdot \omega \in W_n^{\gamma}$, then $r(I)=s(s_I(\gamma\vert_n))=s(I)$, whence $s_I$ is a cycle. Also, ${(s_I\gamma\vert_n)}\vert_n =\gamma\vert_n$, so that $\gamma\vert_n=s_I\tau$. Recurrence on this argument shows that $\gamma\vert_n=s_Is_I\cdots s_I$.

Now, if $s_I$ has no exits, then any $\omega\in X_{p_{r(I)}}$ must be $s_Is_Is_I\cdots$, and so we have equality. Conversely, if $s_I$ have an exit, then there exists a multiindex $K\ne I$ such that $s_Is_K\ne 0$. Then, if $r(K)=j$, for any $\delta \in X_{p_{r(K)}}$ we have that $(\gamma\vert_n)s_K\cdot \delta \in W_n^{\gamma}$ but not in $\alpha_{s_I}(W_n^{\gamma})$, so we are done.
\end{proof}

Since $\alpha_{s_I^*}$ is a partial inverse of $\alpha_{s_I}$, we conclude the following

\begin{lemma}\label{Lem:case s_I^*}
Let $n\in \N$, $\gamma\in X_A$. Then, $\alpha_{s_I^*}(W_n^{\gamma})\subseteq W_n^{\gamma}$ implies that $s_I$ is a cycle with no exits, and so the above inclusion is always an equality.
\end{lemma}

\begin{proposition}\label{Prop: LocContr}
If every finite path in the graph $E_A$ can be enlarged to a cycle and $E_A$ satisfies Condition (L), then $\mathcal{G}_{\LAB}$ is locally contracting.
\end{proposition}
\begin{proof}
Let $n\in \N$ and $\gamma\in X_A$. By hypothesis there exists a finite path $\delta$ $E_A$ with $s(\delta)=r(\gamma\vert_n)$ and $r(\delta)=s(\gamma\vert_n)$. If the length of $\delta$ is $m$, we have $s=\omega=(\gamma\vert_n)\delta$ is a cycle and $W_{n+m}^{\omega}\subset W_n^{\gamma}$. Since $E_A$ satisfies Condition (L), $s$ has an exit, so that by Lemma \ref{Lem:case s_I} $\alpha_s(W_{n+m}^{\omega})\subsetneq W_{n+m}^{\omega}$. Thus, the results holds by (\ref{CondLocContr}).
\end{proof}

So, we can prove the main result of this section

\begin{theorem}\label{Thm: s.p.i.}
Consider the initial matrices $A,B$. If
\begin{enumerate}
\item The matrix $A$ is irreducible.
\item The graph $E_A$ satisfies Condition (L).
\item The matrix $B$ satisfies Condition (E).
\item For any fixed point  $\omega=s_{i_1,i_2,n_1} s_{i_2,i_3,n_2}\cdots s_{i_k,i_{k+1},n_k}\cdots$ and every $n\geq 1$ there exist $m\geq n$ and $j_{m+1}$ with:
\begin{enumerate}
\item $(i_m, j_{m+1})\in \OmA$ .
\item $K_{m+1}=K_m\cdot \displaystyle\frac{B_{i_m, j_{m+1}}}{A_{i_m, j_{m+1}}}\not\in \Z$.
\end{enumerate}
\end{enumerate}
then $\OAB$ is purely infinite simple.
\end{theorem}
\begin{proof}
By Theorem \ref{Thm:Simple} the algebra $\OAB$ is simple. Since the matrix $A$ is irreducible, the graph $E_A$ is transitive (i.e. there exists a path between any two vertices of $E_A$). Thus, $\mathcal{G}_{\LAB}$ is locally contracting by Proposition \ref{Prop: LocContr}. By \cite[Proposition 2.4]{A} every nonzero hereditary sub-$C^*$-algebra of $C_{r}^*(\mathcal{G}_{\LAB})$ contain an infinite projection, and then so does $\OAB$ by Theorem \ref{Prop: amenable}. Hence, $\OAB$ is purely infinite simple, as desired.
\end{proof}

Hence, under the hypotheses of either Theorem \ref{Thm:Simple}, Proposition \ref{Prop:Simple} or Corollary \ref{Corol:Simple}, we conclude that $\OAB$ is purely infinite simple. This includes the case considered by Katsura, when Condition (E) is satisfied. Also, since $A$ irreducible plus Condition (L) implies Condition (K), Theorem \ref{Thm: s.p.i.} becomes an extension of \cite[Theorem 16.2]{ExelLaca} to the case of $B$ being a nonzero matrix.\vspace{.2truecm}

Finally, we will show that, under Condition (E), it is possible to show a partial version of \cite[Proposition 4.5]{Kat2}.

\begin{lemma}\label{Lem:Prop4.5Kat2}
Let $G_0, G_1$ be finitely generated abelian groups. Then, there exist $N\in \N$, $A\in M_N(\Z^+), B\in M_N(\Z)$ such that:
\begin{enumerate}
\item The matrix $B$ satisfies Condition (E).
\item The matrix $A$ is irreducible.
\item $A_{i,i}\geq 2$ and $B_{i,i}=1$ for every $1\leq i\leq N$.
\item $G_0\cong \mbox{coker}(I-A)\oplus \mbox{ker}(I-B)$ and
$G_1\cong \mbox{coker}(I-B)\oplus \mbox{ker}(I-A)$.
\end{enumerate}
\end{lemma}
\begin{proof}
By using the results in \cite[Section 3]{Kat1}, we can guarantee that there exists $N'\in \N$, $A'\in M_{N'}(\Z^+), B'\in M_{N'}(\Z)$ satisfying (2-4). We can choose $A', B'$ such that the Smith Normal Form of both $I'-A'$ and $I'-B'$ are nonzero. Hence, by performing elementary row and column operations, we can replace these matrices by $A''\in M_{N'}(\Z^+), B''\in M_{N'}(\Z)$ such that $0\ne A''_{i,j}, I'-B''_{i,j}$ for every $1\leq i,j\leq N'$. Thus, $A'', C:=I'-B''$ satisfies (1) and (2), while:
\begin{enumerate}
\item $\mbox{coker}(I-A)\cong \mbox{coker}(I'-A'')$ and $\mbox{ker}(I-A)\cong \mbox{ker}(I'-A'')$
\item $\mbox{coker}(C)\cong \mbox{coker}(I'-B'')$ and $\mbox{ker}(C)\cong \mbox{ker}(I'-B'')$
\end{enumerate}

Now define $N=2N'$, and define matrices $A\in M_N(\Z^+), B\in M_N(\Z)$ as follows:
$$
A=\left(
\begin{array}{cc}
 2I' & A''  \\
 I' &   2I'
\end{array}
\right)\hspace{.5truecm}  \mbox{ and } \hspace{.5truecm}
B=\left(
\begin{array}{cc}
 I' & C  \\
 I' &   I'
\end{array}
\right).
$$
Since $A''$ is irreducible, the graph $E_{A''}$ is transitive. It is easy to see that then so is $E_A$, whence $A$ is irreducible. Thus, $A$ and $B$ satisfy (1-3). On one side
$$
I-A=\left(
\begin{array}{cc}
 I' & I'  \\
 0 &   I'
\end{array}\right)\cdot \left(\begin{array}{cc}
 A''-I' & 0  \\
 0 &   -I'
\end{array}\right) \cdot \left(\begin{array}{cc}
 0 & -I'  \\
 I' &   I'
\end{array}\right).$$
Hence, $\mbox{coker}(I-A)\cong \mbox{coker}(I'-A')$ and $\mbox{ker}(I-A)\cong \mbox{ker}(I'-A')$. On the other side
$$
I-B=\left(
\begin{array}{cc}
 C & 0  \\
 0 &   I'
\end{array}\right)\cdot \left(\begin{array}{cc}
 0 & -I'  \\
 -I' &   0
\end{array}\right).$$
Hence, $\mbox{coker}(I-B)\cong \mbox{coker}(I'-B')$ and $\mbox{ker}(I-B)\cong \mbox{ker}(I'-B')$.
Thus, (4) is fulfilled, so we are done.
\end{proof}

As a consequence of Theorem \ref{Thm: s.p.i.} and Lemma \ref{Lem:Prop4.5Kat2}, we conclude the following restricted version of \cite[Proposition 4.5]{Kat2}.

\begin{theorem}[{c.f. \cite[Proposition 4.5]{Kat2}}]\label{Thm:Prop4.5Kat2}
Let $G_0, G_1$ be finitely generated abelian groups. Then, there exist $N\in \N$, $A\in M_N(\Z^+), B\in M_N(\Z)$ satisfying Condition (E), such that:
\begin{enumerate}
\item $\OAB$ is unital Kirchberg algebra.
\item $K_i(\OAB)\cong G_i$ for $i=0,1$.
\end{enumerate}
\end{theorem}

Notice that Theorem \ref{Thm:Prop4.5Kat2} means that we can represent any unital Kirchberg algebra (up to isomorphism) as a Katsura algebra $\OAB$ such that the matrix $B$ satisfies Condition (E), and thus as the groupoid $C^*$-algebra of a minimal essentially principal locally contracting groupoid $\mathcal{G}_{\LAB}$.

\section*{Acknowledgments}

Parts of this work were done during visits of both authors to the Centre de Recerca Matem\`atica (UAB, Spain) and to the Mathematisches Forschungsinstitut Oberwolfach (Germany), and during a visit of the second author to the Departamento de Matem\'atica da Universidade Federal de Santa Catarina (Florian\'opolis, Brasil). The authors thank these host centers for their warm hospitality.

\end{document}